\documentclass[preprint]{elsarticle}
\usepackage[T1]{fontenc}
\usepackage[utf8]{inputenc}
\usepackage{amssymb,amsfonts,amsthm, amsmath}
\usepackage{mathrsfs}
\usepackage{bm}
\usepackage{mathtools}
\usepackage{wasysym}
\usepackage{stmaryrd}
\usepackage{tikz-cd}
\usepackage{svg}
\usepackage{float}
\usepackage[hidelinks]{hyperref}
\usepackage{cleveref}
\usepackage{xcolor}
\usepackage{colortbl}
\usepackage[margin=3.5cm]{geometry}

\date{}
\definecolor{Gray}{gray}{0.95}

\newcommand{\stress}{\sigma}
\newcommand{\fsource}{f}
\newcommand{\bodyforce}{f_u}
\newcommand{\mub}{\boldsymbol{\mu}}
\newcommand{\bias}{\mathsf{b}}
\newcommand{\scl}[2]{\langle #1, #2\rangle}
\newcommand{\sigmaMRE}{\Sigma_h\text{-}\MRE}
\newcommand{\uMRE}{U_h\text{-}\MRE}
\newcommand{\rMRE}{R_h\text{-}\MRE}

\DeclareMathOperator{\asym}{asym}
\DeclareMathOperator{\dev}{dev}
\DeclareMathOperator{\Tr}{Tr}
\DeclareMathOperator*{\argmin}{\arg\!min}
\DeclareMathOperator{\MRE}{MRE}
\DeclareMathOperator{\ACV}{ACV}

\newtheorem{theorem}{Theorem}
\newtheorem{definition}[theorem]{Definition}
\newtheorem{lemma}[theorem]{Lemma}
\newtheorem{example}[theorem]{Example}
\newtheorem{remark}[theorem]{Remark}

\numberwithin{equation}{section}
\numberwithin{theorem}{section}

\begin{document}

\begin{frontmatter}

\title{Neural network solvers for parametrized elasticity problems that conserve linear and angular momentum}

\affiliation[NORCE]{organization={NORCE Norwegian Research Centre},
            addressline={Nygårdsgaten 112}, 
            city={Bergen},
            postcode={5008}, 
            country={Norway}}
\affiliation[MOX]{organization={MOX Laboratory, Department of Mathematics, Politecnico di Milano},
            addressline={Piazza Leonardo da Vinci}, 
            city={Milan},
            postcode={20133}, 
            country={Italy}}
            
\author[NORCE]{Wietse M. Boon}
\author[MOX]{Nicola R. Franco} 
\author[MOX]{Alessio Fumagalli}

\begin{abstract}
    We consider a mixed formulation of parametrized elasticity problems in terms of stress, displacement, and rotation. The latter two variables act as Lagrange multipliers to enforce conservation of linear and angular momentum. Due to the saddle-point structure, the resulting system is computationally demanding to solve directly, and we therefore propose an efficient solution strategy based on a decomposition of the stress variable. First, a triangular system is solved to obtain a stress field that balances the body and boundary forces. Second, a trained neural network is employed to provide a correction without affecting the conservation equations. The displacement and rotation can be obtained by post-processing, if necessary. The potential of the approach is highlighted by three numerical test cases, including a non-linear model.
\end{abstract}



\end{frontmatter}

\section{Introduction}

The development of reliable numerical solution strategies is crucial to accurately model and predict the behavior of elastic materials in engineering applications. 
However, the computational burden of solving the elasticity equations becomes prohibitive in many-query scenarios where the PDE needs to be solved repeatedly for different system configurations, including changes in the material laws, problem parameters, or boundary conditions. 
Instead, a common strategy is to replace expensive numerical simulations with a more efficient surrogate, known as a Reduced Order Model (ROM) \cite{hesthaven_pagliantini_rozza_2022,quarteroni2016reduced}. In this work, we focus on data-driven, non-intrusive ROMs using Deep Learning techniques, which have shown promising potential \cite{kutyniok,rozzarev,franco2023deep}.

Because these approaches are entirely data-driven, they often produce solutions that violate underlying physical principles such as the conservation of mass, momentum, or energy. We herein adopt the viewpoint that, while some error in constitutive laws are tolerable, physical conservation laws need to be satisfied precisely. 

A variety of deep-learning methods have been developed to ensure conservation laws in PDE modeling. Physics-Informed Neural Networks \cite{Raissi2019686,rozzarev} enforce these constraints through the loss function. However, this approach often proves insufficient as the constraints are not guaranteed to be satisfied exactly \cite{Hansen2023}. In contrast, exact constraint-preserving networks are explored in \cite{boesen2022neural,ruthotto2020deep}, where the authors studied the incorporation of supplementary information into neural networks for modeling dynamical systems.
In \cite{trask2022enforcing}, a data-driven exterior calculus is introduced to produce structure-preserving ROMs that adhere to physical laws. Finally, \cite{beucler2021enforcing} proposed enforcing analytical constraints by mapping the solution onto the kernel of the constraint operator. In this work, we follow a similar strategy and develop such projection operators explicitly.

In \cite{Boon2022,boon2023deep}, a strategy is proposed to construct solvers that guarantee mass conservation for flow problems. This work effectively extends those ideas to the case of parameterized elasticity problems. We focus on a mixed formulation of the elasticity equation that explicitly includes the linear and angular momentum balance equations. We discretize these equations by employing the low-order mixed finite element triplet from \cite{arnold2007mixed} to model the stress, displacement, and rotation variables.

Our method involves decomposing the stress tensor into two parts: a particular solution that balances the body and boundary forces, and a homogeneous correction that conserves linear and angular momentum locally. We approximate these components separately. For the particular solution, we propose an efficient solution procedure based on a spanning tree in the mesh. For the remainder, we consider suitable neural network architectures combined with a kernel projector, which ensures that the output remains within the null space of the constraint operator. Through this procedure, we ensure that the stress field satisfies the conservation equations up to machine precision. Finally, the spanning tree can be employed again to efficiently post-process the displacement and rotation fields.

The paper is structured as follows. In \Cref{sec:model}, we introduce the elastic models compatible with our approach, as well as the mixed finite element discretization. \Cref{sec: spanning tree} is dedicated to the construction of a spanning tree solver, an ingredient of fundamental importance for our construction. Next, the relevant background and preliminaries on deep learning algorithms are presented in \Cref{sec:dl preliminaries}. \Cref{sec:conservative solvers} combines the spanning tree and deep learning techniques to form our reduced order models. Numerical experiments, designed to compare performance across problems of increasing complexity, are discussed in \Cref{sec:experiments}. Finally, concluding remarks are given in \Cref{sec:conclusion} and we provide additional details concerning Proper Orthogonal Decomposition and the used neural networks in \Cref{appendix: POD}.

\section{Model problems}
\label{sec:model}

Let $\Omega \in \mathbb{R}^d$ with $d = 2, 3$ be a Lipschitz, polytopal domain and let its boundary be disjointly partitioned as $\partial \Omega = \partial_u \Omega \cup \partial_\sigma \Omega$ with $|\partial_u \Omega| > 0$.
We consider a mixed formulation of elasticity problems in which the stress field is a primary variable. This allows us to explicitly impose the balance of linear and angular momentum as equations in the system. We thus consider the problem unknowns
\begin{align}
    \sigma&: \Omega \to \mathbb{R}^{d \times d}, & 
    u&: \Omega \to \mathbb{R}^d, & 
    r&: \Omega \to \mathbb{R}^{\binom{d}{2}}.
\end{align}
in which $\sigma$ denotes the Cauchy stress tensor, $u$ is the displacement, and $r$ is a variable commonly referred to as the \emph{rotation}. Note that $r$ is a scalar variable in 2D and a vector field in 3D. The rotation acts as a Lagrange multiplier to enforce symmetry of $\sigma$ through the operator $\asym: \mathbb{R}^{d \times d} \to \mathbb{R}^{\binom{d}{2}}$. This operator, and its adjoint $\asym^*$, are given by:
\begin{align*}
    \asym \sigma               & = \begin{bmatrix}
        \sigma_{32} - \sigma_{23} \\
        \sigma_{13} - \sigma_{31} \\
        \sigma_{21} - \sigma_{12}
    \end{bmatrix}, &
    \asym^* r & =
    \begin{bmatrix}
        0    & -r_3 & r_2  \\
        r_3  & 0    & -r_1 \\
        -r_2 & r_1  & 0
    \end{bmatrix}, &
    d = 3,                                                     \\\\
    \asym \sigma               & = \sigma_{21} - \sigma_{12}, &
    \asym^* r                  & =
    \begin{bmatrix}
        0 & -r \\
        r & 0
    \end{bmatrix}, &
    d = 2.
\end{align*}

We are interested in elasticity problems of the form: find $(\sigma, u, r)$ such that
\begin{subequations}
    \label{eq:model-problem}
    \begin{align}
        A\sigma - \nabla u - \asym^*r & = 0,  \label{eq: stress-strain} \\
        -\nabla \cdot \sigma          & = \bodyforce,  \label{eq: lin momentum} \\
        \asym \sigma                  & = 0,  \label{eq: ang momentum} 
    \end{align}
    on $\Omega$, subject to the boundary conditions
    \begin{align}    
        u &= g_u\;\;\text{on}\;\partial_u \Omega, &
        \nu \cdot \sigma &= 0\;\;\text{on}\;\partial_\sigma \Omega.
    \end{align}
\end{subequations}
Here, $\bodyforce$ is a given body force, $g_u$ prescribes the displacement on the boundary $\partial_u \Omega$, and $\nu$ is the outward unit normal of $\partial \Omega$. For clarity, we moreover note that $\nabla u$ denotes the gradient of the displacement and $\nabla \cdot \sigma$ refers to the row-wise divergence on the Cauchy stress.

We briefly elaborate on the physical meaning of the equations of \eqref{eq:model-problem}. First, the operator $A: \mathbb{R}^{d \times d} \to \mathbb{R}^{d \times d}$ in \eqref{eq: stress-strain} is the (possibly non-linear), stress-strain relationship. Depending on the material properties, different laws can be considered, resulting in different choices, and parametrizations, of $A$. We report two examples below.

\begin{example}[Hooke's law]\label{ex: hook}
    In the case of linearized elasticity and an isotropic, homogeneous material, the stress-strain relation is given by
    \begin{align} \label{eq: Hooke}
        A\sigma \coloneqq \frac1{2\mu}\left( \sigma - \frac{\lambda}{2\mu + d\lambda} \Tr(\sigma)I \right),
    \end{align}
    in which $\mu$ and $\lambda$ are the Lamé parameters, $\Tr$ denotes the matrix trace and $I$ is the identity tensor. The inverse of $A$ yields the recognizable expression:
    \begin{align*}
        A^{-1} \varepsilon = 2 \mu \varepsilon + \lambda \Tr (\varepsilon) I.
    \end{align*}
\end{example}

\begin{example}[Hencky-von Mises] \label{ex: hencky-von mises}
    As investigated in \cite{gatica2013priori,da2015virtual}, the Hencky-von Mises model defines the stress-strain relationship $A$ as in \eqref{eq: Hooke}, but the Lamé parameters are modeled as functions of the deviatoric strain, i.e. 
    \begin{align*}
        \mu &= \mu(\| \dev(\varepsilon) \|), & 
        \lambda &= \lambda(\| \dev(\varepsilon) \|).
    \end{align*}
    Here, $\varepsilon=\frac{1}{2}(\nabla u+(\nabla u)^\top)$ is the linearized strain tensor, $\| \cdot \|$ denotes the Frobenius norm, and $\dev$ denotes the deviator
    $\dev(\tau) \coloneqq \tau - \frac1d \Tr(\tau)$.
    We emphasize that the operator $A$ is non-linear in this case.
\end{example}

The conservation laws are given by \eqref{eq: lin momentum} and \eqref{eq: ang momentum}. The first conservation law describes the balance of linear momentum. By choosing appropriate finite element spaces (cf. \Cref{sub: FOM}), we will impose this law strongly and pointwise in $\Omega$.

The second conservation law is the balance of angular momentum which, under the assumption of linear momentum conservation, is equivalent to imposing the symmetry of the Cauchy stress tensor, i.e. \eqref{eq: ang momentum}. Due to the difficulties in constructing stable mixed finite elements that exactly satisfy this constraint, we herein follow \cite{arnold2007mixed} and impose the symmetry constraint weakly.

We emphasize that we only consider models that feature linear conservation laws \eqref{eq: lin momentum} and \eqref{eq: ang momentum}, while allowing for non-linearities in the constitutive law \eqref{eq: stress-strain}.

\begin{remark}[Hyperelasticity]
    In the case of hyperelasticity, the infinitesimal strain $\varepsilon$ is replaced by the Lagrangian Green strain $E \coloneqq \varepsilon + \frac12(\nabla u)^\top (\nabla u)$. Additionally, the second Piola-Kirchhoff stress tensor is given by
    \begin{align*}
        S = 2\mu E + \lambda Tr(E) I,
    \end{align*}
    The momentum balance and symmetry constraints are now given by
    \begin{align*}
        - \nabla \cdot (F S) & = \bodyforce, &
        \asym S              & = 0,
    \end{align*}
    in which $F$ is the deformation gradient. Note that the equation describing linear momentum has become non-linear. Changing variables to the first Piola-Kirchhoff stress tensor $P \coloneqq F S$ does not improve matters because this would introduce a non-linearity in the symmetry constraint instead. The approach from \cite{boon2023deep}, which we follow herein, is only valid for linearly constrained systems and thus hyperelasticity falls outside the scope of this work.
\end{remark}

As mentioned in the introduction, we aim to construct efficient numerical solvers that, given a problem of interest, can rapidly produce multiple simulations for varying parameters (e.g. different values of the Lamé parameters $\mu,\lambda$, perturbations in the body force $\bodyforce$, or changes in the boundary conditions $g_u$ for the displacement). Specifically, we are interested in surrogate models based on neural networks that produce accurate simulations after training on a set of pre-computed solutions $\{(\sigma_i,u_i,r_i)\}_i$, obtained for various configurations of the system. 

We will only employ neural networks in the solving stages of the problem. This means that we will specify the finite element basis beforehand and propose a solver to produce the coefficients. We therefore dedicate the next subsection to a brief explanation of the chosen finite element spaces.

\subsection{A mixed finite element discretization}
\label{sub: FOM}

Let the mesh $\Omega_h$ be a shape-regular, simplicial tesselation of $\Omega$. For the discretization of \eqref{eq:model-problem}, we will consider the finite element triplet of lowest order proposed in \cite{arnold2006finite,arnold2007mixed}. These spaces are given by (tuples of) the Brezzi-Douglas-Marini element $\mathbb{BDM}_1$ and the piecewise constants $\mathbb{P}_0$ on $\Omega_h$:
\begin{subequations}
\begin{align} \label{eq: discrete spaces}
    \Sigma_h & \coloneqq \{ \sigma_h \in \mathbb{BDM}_1^d(\Omega_h) : \nu \cdot \sigma = 0 \text{ on } \partial_\sigma \Omega \},     \\
    U_h      & \coloneqq \mathbb{P}_0^d(\Omega_h),            &
    R_h      & \coloneqq \mathbb{P}_0^{\binom{d}{2}}(\Omega_h)
\end{align}
\end{subequations}
The finite element problem then becomes: find $(\sigma_h, u_h, r_h) \in \Sigma_h \times U_h \times R_h$ such that
\begin{subequations}
    \label{eq:finite-element}
    \begin{align}
        \scl{A\sigma_h}{\tilde \sigma_h}_\Omega + \scl{u_h}{\nabla \cdot \tilde \sigma_h}_\Omega - \scl{r_h}{\asym \tilde \sigma_h}_\Omega                                                 & = \scl{g_u}{\nu \cdot \tilde \sigma_h}_{\partial_u \Omega} \\
        -\scl{\nabla \cdot \sigma_h}{\tilde u_h}_\Omega & = \scl{\bodyforce}{\tilde u_h}_\Omega                                 \\
        \scl{\asym \sigma_h}{\tilde r_h}_\Omega         & = 0,
    \end{align}
\end{subequations}
for all test functions $(\tilde\sigma_h,\tilde u_h,\tilde r_h)\in\Sigma_h\times U_h\times R_h$. Here $\scl{\phi}{\psi}_\Omega \coloneqq \int_{\Omega}\phi(x)\psi(x)dx$ and $\scl{\phi}{\psi}_{\partial_u\Omega} \coloneqq \int_{\partial_u\Omega}\phi(s)\psi(s)ds$ denote the relevant $L^2$ inner products. 
It was shown in \cite{arnold2007mixed} that, assuming lower and upper bounds on $A$, Problem \eqref{eq:finite-element} admits a unique solution that is bounded in the following norms:
\begin{subequations}\label{eqs:norms}
\begin{align} 
    \| \sigma_h \|_{\Sigma_h}^2 &\coloneqq 
    \| \sigma_h \|_{L^2(\Omega)}^2 +
    \| \nabla \cdot \sigma_h \|_{L^2(\Omega)}^2, \\
    \| (u_h, r_h) \|_{X_h}^2 &\coloneqq 
    \| u_h \|_{U_h}^2\;\;\;\;+
    \| r_h \|_{R_h}^2
    \nonumber
    \\
    &\coloneqq 
    \| u_h \|_{L^2(\Omega)}^2 +
    \| r_h \|_{L^2(\Omega)}^2.
\end{align}
\end{subequations}

When $A$ is linear, e.g.~as in \Cref{ex: hook}, problem \eqref{eq:finite-element} is a linear system in the unknowns $(\sigma_h,u_h,r_h)$ that can be addressed with a direct solver. However, if $A$ is non-linear, then iterative solvers are required, and we propose a specific scheme in the following example.

\begin{example}[Iterative solver for the Hencky-von Mises model] \label{ex: hencky-von mises solver}
\Cref{ex: hencky-von mises} assumes that the Lamé parameters depend on the deviatoric strain. 
In our discrete formulation, we do not immediately have access to the strain (nor its deviator) since the displacement $u_h$ is only given as a piecewise constant. However, we can retrieve the deviatoric strain by applying the deviator on the stress-strain relationship:
\begin{align*}
    \dev(\sigma) = 2\mu \dev(\varepsilon) + \lambda \Tr(\varepsilon) \dev(I)
    = 2\mu \dev(\varepsilon),
\end{align*}
and thus, $\| \dev(\sigma) \|
    = 2 \mu \| \dev(\varepsilon) \|.$
Using this observation, we will solve \eqref{eq:finite-element} using the following iterative scheme:
\begin{enumerate}
    \item Set $i = 0$ and compute an initial guess $(\sigma_0,u_0,r_0)\in\Sigma_h\times U_h\times R_h$, e.g. by solving \eqref{eq:finite-element} with constant Lamé parameters.
    \item Evaluate $\sigma_i$ at the cell centers and compute the norm of its deviator.
    \item In each cell, solve the following non-linear equation for $0 \le \zeta \in \mathbb{R}$:
    \begin{equation*}
        2 \mu(\zeta) \zeta
        = \| \dev(\sigma_i) \|.
    \end{equation*}
    \item Evaluate $\mu=\mu(\zeta)$, $\lambda=\lambda(\zeta)$, then solve \eqref{eq:finite-element} to obtain $(\sigma_{i+1},u_{i+1},r_{i+1})$. Increment $i \leftarrow i + 1$.
    \item Iterate steps 2-4 until convergence.
\end{enumerate}
\end{example}

\section{A spanning tree solver for the stress}
\label{sec: spanning tree}

We now leverage the mathematical structure of the Full-Order Model (FOM) discretization \eqref{eq:finite-element} to derive an efficient numerical solver that, despite operating at FOM level, handles the conservation of linear and angular momentum in linear time. This will be a fundamental ingredient for our construction of reduced order models in \Cref{sec:conservative solvers}.

To start, we introduce the following short-hand notation for the \textit{constraint operator} $B: \Sigma_h \to U_h' \times R_h'$ and the functional $f \in U_h' \times R_h'$:
\begin{align*}
    \scl{B \sigma_h}{(\tilde u_h, \tilde r_h)} &\coloneqq 
    -\scl{\nabla \cdot \sigma_h}{\tilde u_h}_\Omega
    + \scl{\asym \sigma_h}{\tilde r_h}_\Omega, \\
    \scl{f}{(\tilde u_h, \tilde r_h)} 
    &\coloneqq \scl{f_u}{\tilde u_h}_\Omega
\end{align*}
in which $U_h'$ and $R_h'$ are the dual spaces of $U_h$ and $R_h$, respectively, and the angled brackets on the left-hand side indicate duality pairings. The conservation laws \eqref{eq: lin momentum}-\eqref{eq: ang momentum} can then be expressed as
\begin{equation}
    \label{eq: constraint}
    B\sigma = f,
\end{equation}
For parametric problems, we adopt a subscript $f_{\mub}$ to indicate the dependency on a set of parameters collected in the vector $\mub$. We emphasize that the operator $B$ is linear and independent of $\mub$. 

As discussed in \cite{boon2023deep}, we can handle linear constraints such as \eqref{eq: constraint} by constructing a right-inverse of the constraint operator. In particular, we construct an operator $S_I$ that, for any given $f \in U_h' \times R_h'$ produces a $S_I f \in \Sigma_h$ such that
\begin{align*}
    B S_I f = f.
\end{align*}

To construct $S_I f$ for a given $f \in U_h' \times R_h'$, we first consider the surjectivity of the operator $B$. For that, we recall the following result. 
\begin{lemma} \label{lem: infsup standard}
The discrete spaces given by \eqref{eq: discrete spaces} satisfy the following condition:
\begin{align*}
    \inf_{(u_h, r_h) \in U_h \times R_h}
    \;\sup_{\sigma_h \in \Sigma_h}\;
    \frac{\langle B\sigma_h, (u_h, r_h) \rangle}{\| \sigma_h \|_{\Sigma_h} \| (u_h, r_h) \|_{X_h}} \ge C > 0
\end{align*}
with $C$ independent of the mesh size $h$.
\end{lemma}
\begin{proof}
    See \cite[Thm.~7.1]{arnold2007mixed}.
\end{proof}

The inf-sup condition of \Cref{lem: infsup standard} implies that if $\langle B\sigma_h, (u_h, r_h) \rangle = 0$ for all $\sigma_h \in \Sigma_h$, then $(u_h, r_h) = (0, 0)$. In other words, $B$ is surjective.

We now briefly consider the edge case where an element $\omega$ borders multiple boundaries on which essential boundary conditions are imposed, such that only one of its facets,  $\gamma$, contains degrees of freedom in the space $\Sigma_h$. Since the proof of \Cref{lem: infsup standard} covers this case, this means that the operator $B$ restricted to this single facet $\gamma$ remains surjective on the space restricted to $\omega$. 
In turn, we may construct $S_I f$ by moving through the grid one facet and cell at a time. We do this by constructing a \emph{spanning tree} as follows. 

We start by considering the dual grid of $\Omega_h$ as a graph. In particular, we construct a graph $\mathcal{G}$ by assigning a node to each cell $\omega \in \Omega_h$ and we connect two nodes with an edge if the corresponding cells share a facet. Next, we choose a cell $\omega_0$ that borders the boundary $\partial_u \Omega$ and an adjacent facet $\gamma_0$ on that boundary. We then construct a spanning tree $\mathcal{T}$ of the graph $\mathcal{G}$ rooted at the node corresponding to $\omega_0$. An example is illustrated in \Cref{fig:spt}(left). For readers unfamiliar with the concept, we recall that a spanning tree is a connected subgraph that: i) contains all the vertices of the original graph, ii) has no cycles. These structures are often employed to re-organize data on graphs. For an example in the context of finite elements, see \cite{de2023construction}.

\begin{remark}[Multiple roots]
    Alternatively, we may choose multiple roots near the boundary to construct multiple trees that, together, form a spanning forest that reaches every cell in the grid exactly once. An example is illustrated in \Cref{fig:spt}(right). 
\end{remark}
\begin{figure}[ht]
    \centering
    \includegraphics[width=0.33\linewidth]{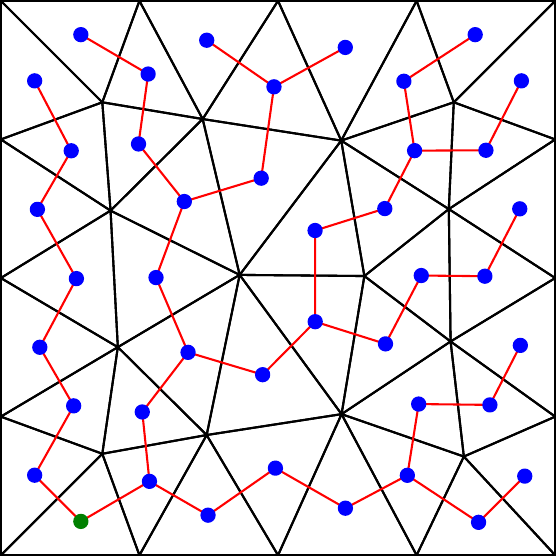}%
    \hspace{0.2\textwidth}%
    \includegraphics[width=0.33\linewidth]{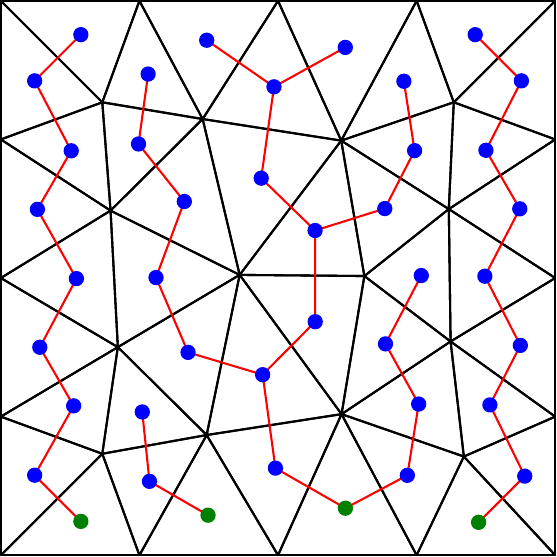}
    \caption{Example of spanning trees for a two-dimensional grid, on the left with single root (green dot) and on the right with multiple roots. Each root is associated to a facet on the boundary on which a displacement boundary condition is imposed.}
    \label{fig:spt}
\end{figure}

The edges of the tree $\mathcal{T}$ correspond to a subset of facets in the grid $\Omega_h$ and, together with the root facet $\gamma_0$, we refer to this set as $\mathcal{F}_\mathcal{T}$. Note that this set contains as many facets as there are cells in the grid. This is due to the fact that the number of nodes in a tree is one higher than its number of edges, and we have included the boundary facet in $\mathcal{F}_\mathcal{T}$. 

We now combine the surjectivity of $B$ with the spanning tree $\mathcal{T}$. In particular, we note that the degrees of freedom of $\mathbb{BDM}_1^d$ on $\mathcal{F}_\mathcal{T}$ are sufficient to construct a right-inverse of $B$. However, on each facet, we have 9 degrees of freedom (4 in 2D) and 6 equations to fulfill per cell (3 in 2D). Indeed, the space $\mathbb{BDM}_1^d$ contains more degrees of freedom than necessary for the surjectivity to hold.
This was also observed in \cite[Sec. 8]{arnold2007mixed} and, consequently, a reduced space was proposed, which we denote by $\mathbb{AFW}^-$. This space has 6 degrees of freedom per facet in 3D, respectively 3 in 2D, and therefore matches exactly the number of equations we need to satisfy. The explicit degrees of freedom can be found in \cite[Sec. 8]{arnold2007mixed} and \cite[Sec. 11.8]{arnold2006finite}.

As shown in \cite{arnold2006finite,arnold2007mixed}, the analogue of \Cref{lem: infsup standard} holds in the reduced space, namely:
\begin{align*}
    \inf_{(u_h, r_h) \in U_h \times R_h}
    \;\sup_{\sigma_h \in \mathbb{AFW}^-}\;
    \frac{\langle B\sigma_h, (u_h, r_h) \rangle}{\| \sigma_h \|_{\Sigma_h} \| (u_h, r_h) \|_{X_h}} \ge c > 0
\end{align*}

Let now $\Pi: \mathbb{BDM}_1^d(\Omega_h) \to \mathbb{AFW}^-(\mathcal{F}_\mathcal{T})$ be the restriction onto those degrees of freedom of the reduced space that are on the facets of $\mathcal{F}_\mathcal{T}$. Its adjoint $\Pi^*$ is then an inclusion operator into $\mathbb{BDM}_1^d$.

Following the same arguments as above, the operator $B \Pi^*$ is surjective as well. Let $\mathsf{B \Pi^T} \in \mathbb{R}^{n \times n}$ be its matrix representation with $n = \dim(U_h) + \dim(R_h) = \dim(\mathbb{AFW}^-(\mathcal{F}_\mathcal{T}) = 6\cdot n_{cells}$, respectively $3 \cdot n_{cells}$ in 2D. Since it is a square, surjective matrix, it is invertible, which allows us to define:
\begin{align*}
    S_I \coloneqq \Pi^* (B \Pi^*)^{-1}
\end{align*}
It is easy to see that $S_I$ is a right-inverse of $B$. This operator will be the key building block in our construction because it provides a weakly-symmetric stress field that balances the forces. In turn, we only need to provide a correction that lies in the kernel of $B$, for which we will use trained neural networks.

\begin{remark}
    The application of $S_I$ involves solving a sparse linear system. Due to the underlying spanning tree structure, this linear system is equivalent to a triangular matrix. In particular, we may solve the equations relevant to the leaf nodes first, which decides the degrees of freedom on the adjacent face. We can then progress iteratively through the tree towards the root. Since there are only 6 equations per cell, this system can be solved in linear time with respect to the number of cells.
\end{remark}

We conclude this section by noting that $S_I$ can be used to reconstruct the displacement and rotation variables. In particular, for given stress $\sigma_h$, we can use the adjoint $S_I^*: \Sigma_h' \to U_h \times R_h$ and \eqref{eq: stress-strain} to derive
\begin{align} \label{eq: postprocessing}
    (u_h, r_h)
     & = (B S_I)^* (u_h, r_h)
    = S_I^* (B^*(u_h, r_h))
    = S_I^* (A_h \sigma_h - g_h)
\end{align}
in which $A_h: \Sigma_h \to \Sigma_h'$ and $g_h \in \Sigma_h'$ are defined such that:
\begin{align*}
    \scl{A_h \sigma_h}{\tilde \sigma_h} &\coloneqq 
    \scl{A \sigma_h}{\tilde \sigma_h}_\Omega, &
    \scl{g_h}{\tilde \sigma_h} 
    &\coloneqq \scl{g_u}{\nu \cdot \tilde \sigma_h}_{\partial_u \Omega}
\end{align*}
Finally, recall that $S_I$ involves solving a sparse and triangular system. The adjoint operator solves the transpose problem, which is also sparse and triangular, so the post-processing of the displacement and rotation can be done efficiently for given $\sigma_h$, $A_h$, and $g_u$.

\section{Preliminaries on neural network based non-linear regression} 
\label{sec:dl preliminaries}

In this section, we recall some fundamental concepts of deep learning, starting from the definition and training of classical neural network architectures. A brief overview of the POD-NN approach is also presented. Readers familiar with these concepts may skip this section and proceed to \Cref{sec:conservative solvers}.

\subsection{Black-box models: definition and training}
\label{sub: regression}
Feed forward neural networks are computational units commonly employed for non-linear regression. At their core, neural networks are based on the composition of affine transformations and non-linear activations. We provide a rigorous definition for clarity.

\begin{definition}[Layer]
\label{def:layer}
Let $\rho:\mathbb{R}\to\mathbb{R}$ and let $m,n\in\mathbb{N}_+.$ A map $L:\mathbb{R}^{m}\to\mathbb{R}^{n}$ is said to be a layer with activation function $\rho$ if it can be expressed as
\begin{equation*}
    L(\mathsf{y})=\rho(\mathsf{W}\mathsf{y}+b),
\end{equation*}
for a suitable matrix $\mathsf{W}\in\mathbb{R}^{n\times m}$, called weight-matrix, and a suitable vector $b\in\mathbb{R}^{n}$, called bias vector. Here, the action of the activation function $\rho$ is intended component-wise, i.e. $\rho([w_1,\dots,w_m]^\top) \coloneqq [\rho(w_1),\dots,\rho(w_m)]^\top.$ If $\rho$ is the identity, the layer is said to be affine.
\end{definition}

\begin{definition}[Feed forward neural network] Let $\rho:\mathbb{R}\to\mathbb{R}$. A neural network with activation function $\rho$ is any map $\Phi:\mathbb{R}^{m}\to\mathbb{R}^{n}$ that can be expressed as
\begin{equation}
\label{eq: deep neural network}
\Phi(v)=(L_{k}\circ\dots\circ L_0)(v),\end{equation}
where $L_k:\mathbb{R}^{n_k}\to\mathbb{R}^{n}$ is an affine layer, whereas $L_{i}:\mathbb{R}^{n_{i}}\to\mathbb{R}^{n_{i+1}}$ are all layers with activation function $\rho$. Here, $n_0 \coloneqq m$.    
\end{definition}

To illustrate the use of neural network models for non-linear regression, let us assume we are given a map
\begin{equation*}
\mathscr{F}:M\ni \mub\mapsto \mathsf{y}_{\mub}\in\mathbb{R}^{N_h},
\end{equation*}
where $M\subset\mathbb{R}^{p}$ and $\mathbb{R}^{N_h}$ denote the input and output spaces, respectively. In our case, $\mub$ will be the vector of parameters parametrizing the elasticity problem whereas $\mathsf{y}_{\mub}$ is the vector containing the finite element coefficients of the stress field. Consequently, $\mathscr{F}$ corresponds to the full-order model.

We are interested in learning $\mathscr{F}$, i.e., in finding a suitable neural network model that can approximate the map $\mub\mapsto\mathsf{y}_{\mub}.$ We assume that the reference model $\mathscr{F}$ is at our disposal during the training phase.
The classical procedure consists of three steps. First, the \textit{sampling phase} consists of querying $\mathscr{F}$ to construct a training set 
\begin{equation*}
    \{\mub_i,\mathsf{y}_{\mub_i}\}_{i=1}^{N}\subset M\times\mathbb{R}^{N_h}.
\end{equation*}
The second step concerns the design of a suitable neural network architecture
\begin{equation*}
    \Phi:\mathbb{R}^{p}\to\mathbb{R}^{N_h}.
\end{equation*}
This corresponds to choosing the number of layers in \eqref{eq: deep neural network}, together with their input-output dimensions, without  specifying their weight matrices and bias vectors: these are collected in a common list $\theta \coloneqq (\mathsf{W}_0,\bias_0,\dots,\mathsf{W}_k,\bias_k)$ and regarded as tunable parameters.
The final step is the so-called \textit{training phase}, which ultimately consists in solving the following minimization problem
\begin{equation}
\label{eq:training}
\theta_* \coloneqq \argmin_{\theta}\;\frac{1}{N}\sum_{i=1}^{N}\|\mathsf{y}_{\mub_i}-\Phi_\theta(\mub_i)\|_{Y}^2.\end{equation}
Here, $\|\cdot\|_{Y}$ is a given norm defined over the output space $\mathbb{R}^{N_h}$.
The right-hand-side of \eqref{eq:training} is commonly known as \textit{loss function}, here defined using the mean squared error metric. Once the training phase is completed, the non-linear regression model is obtained as $\Phi \coloneqq \Phi_{\theta_*}.$

We mention that, in general, the training phase can be computationally challenging due to i) non-linearity and non-convexity of the loss function and ii) the large number of trainable parameters. To overcome this difficulty, several strategies have been developed. We highlight the so-called \textit{POD-NN approach}, which has the advantage of simplifying the training phase when $N_h$ is large. We provide a quick overview in the next subsection.

\subsection{The POD-NN strategy}
\label{sub: pod-nn}
POD-NN is a technique combining Proper Orthogonal Decomposition (POD), a commonly used dimensionality reduction technique, and neural networks (NN). The approach was originally developed in the context of model order reduction \cite{hesthaven2018non}. More precisely, it was introduced as a data-driven alternative of the so-called reduced basis method \cite{quarteroni2016reduced}. In subsequent years it was re-discovered, and generalized, under different names, such as PCA-net and POD-DeepONet \cite{lanthaler2023operator, lu2022comprehensive}.
From a purely practical point of view, however, the POD-NN approach can regarded as an alternative training strategy applied to classical neural network regression. 

To illustrate the idea, let $\Phi_{\theta}:\mathbb{R}^{p}\to\mathbb{R}^{N_h}$ be a trainable neural network architecture, where $\theta \coloneqq (\mathsf{W}_0,\bias_0,\dots,\mathsf{W}_k,\bias_k)$ as before. Let $n\in\mathbb{N}_+$ be such that $\mathsf{W}_k\in\mathbb{R}^{N_h\times n}$. We can highlight the presence of the last layer by re-writing the action of the neural network model as
\begin{equation*}
\Phi_\theta(\mub)=\mathsf{W}_k\phi_{\tilde\theta}(\mub)+\bias_k,
\end{equation*}
where $\phi_{\tilde\theta}:\mathbb{R}^{p}\to\mathbb{R}^{n}$ is the submodule collecting all layers except for the last one, so that $\tilde{\theta} \coloneqq (\mathsf{W}_0,\bias_0,\dots,\mathsf{W}_{k-1},\bias_{k-1})$. Assume that $\bias_k\equiv0.$ Then, it is straightforward to see that, for every $\mub\in M$, the vector $\Phi_\theta(\mub)\in\mathbb{R}^{N_h}$ lies in the subspace of $\mathbb{R}^{N_h}$ spanned by the columns of $\mathsf{W}_k$. In particular, if $\mathsf{V}\in\mathbb{R}^{N_h\times n}$ is an orthonormal matrix spanning the same subspace of $\mathsf{W}_k$, by optimality of orthogonal projections we have
\begin{equation*}
    \|\mathsf{y}_{\mub}-\Phi_\theta(\mub)\|=\|\mathsf{y}_{\mub}-\mathsf{W}_k\phi_{\tilde\theta}(\mub)\|\le\|\mathsf{y}_{\mub}-\mathsf{V}\mathsf{V}^\top\mathsf{y}_{\mub}\|
\end{equation*}
for all $\mub\in M$. With this observation, the POD-NN approach proposes the following design and training strategy:
\begin{enumerate}
    \item For a given tolerance $\epsilon>0$, find a suitable $n\in\mathbb{N}_+$ and a corresponding matrix $\mathsf{V}\in\mathbb{R}^{N_h\times n}$ such that
    \begin{equation}        \label{eq: pod loss}
        \frac{1}{N}\sum_{i=1}^{N}\|\mathsf{y}_{\mub_i}-\mathsf{V}\mathsf{V}^\top\mathsf{y}_{\mub_i}\|^2<\epsilon^2.
    \end{equation}
    In practice, this can be achieved by leveraging the POD algorithm, which is ultimately based on a truncated singular value decomposition. We refer the interested reader to \Cref{appendix: POD}.
    \item Introduce a trainable network $\phi_{\tilde\theta}:\mathbb{R}^{p}\to\mathbb{R}^{n}$ and train it according to the loss function
    \begin{equation}
        \label{eq: pod-nn loss}
       \mathscr{L}(\tilde\theta) \coloneqq \frac{1}{N}\sum_{i=1}^{N}\|c_{\mub_i}-\phi_{\tilde\theta}(\mub_i)\|^2,
    \end{equation}
    where $c_{\mub} \coloneqq \mathsf{V}^\top\mathsf{y}_{\mub_i}$ are the ideal coefficients over the POD basis.
    \item Define $\Phi=\Phi_{\textnormal{POD}}$ as $\Phi_\textnormal{POD}(\mub) \coloneqq \mathsf{V}\phi(\mub)$, where $\phi=\phi_{\tilde\theta_*}$ is the trained module. 
\end{enumerate}
For the sake of simplicity, we presented the POD-NN strategy using the $\ell^2$-norm $\|\cdot\|$ (Euclidean norm). However, the procedure can be adjusted to account for arbitrary norms $\|\cdot\|_{Y}$ satisfying the parallelogram rule, cf. Remark \ref{remark: norms}.

\begin{remark}
\label{remark: norms}
One can easily adapt the construction in \Cref{sub: pod-nn} to ensure optimality in any norm induced by an inner-product. Given such a norm $\|\cdot\|_{Y}$, let $\mathsf{D}\in\mathbb{R}^{N_h\times N_h}$ be its Gramian matrix, so that $\|\mathsf{y}\|_{Y}^2=\mathsf{y}^\top\mathsf{D}\mathsf{y}$. Then, the POD-matrix $\mathsf{V}$ can be constructed to be $\|\cdot\|_{Y}$-orthonormal, meaning that $\mathsf{V}^\top\mathsf{D}\mathsf{V}=\mathsf{I}_n$: cf. \Cref{appendix: POD}. 

Alternatively, one may use the $\ell^2$-norm to construct $\mathsf{V}$, as in \eqref{eq: pod loss}, but then re-introduce the norm $\|\cdot\|_{Y}$ at the latent level. Note that for $c_{\mub} \coloneqq \mathsf{V}^\top\mathsf{y}_{\mub}$, by the triangular inequality,
\begin{align*}    
    \|\mathsf{y}_{\mub}-\Phi_\theta(\mub)\|_{Y}&\le
    \|\mathsf{y}_{\mub}-\mathsf{V}c_{\mub}\|_{Y}+\|\mathsf{V}c_{\mub}-\mathsf{V}\phi_{\tilde\theta}(\mub)\|_{Y}=\\&=
    \|\mathsf{y}_{\mub}-\mathsf{V}c_{\mub}\|_{Y}+\|c_{\mub}-\phi_{\tilde\theta}(\mub)\|_{\mathsf{D}^{1/2}\mathsf{V}},
\end{align*}
where $\|\cdot\|_{\mathsf{D}^{1/2}\mathsf{V}}$ is the norm over $\mathbb{R}^{n}$ defined via the relation
\begin{equation*}
    \|a\|_{\mathsf{D}^{1/2}\mathsf{V}}^2=a^\top\left(\mathsf{V}^\top\mathsf{D}\mathsf{V}\right)a.
\end{equation*}
In particular, one can: i) choose $n$ and construct $\mathsf{V}$ via classical POD so that, a posteriori, $\|\mathsf{y}_{\mub}-\mathsf{V}\mathsf{V}^\top\mathsf{y}_{\mub}\|_{Y}$ is below a given tolerance, ii) train $\phi_{\tilde\theta}$ by minimizing a weighted loss in which $\|\cdot\|_{\mathsf{D}^{1/2}\mathsf{V}}$ replaces the Euclidean norm in \eqref{eq: pod-nn loss}. In this work, this is the default implementation of the POD-NN strategy.
\end{remark}

\begin{remark}
    \label{remark:homogeneous}
    It can be shown that, regardless of the underlying norm, the POD-matrix is always expressible as linear combination of the training snapshots. Consequently,
    \begin{equation*}
        \Phi_\textnormal{POD}(\mub)\in\textnormal{span}(\mathsf{V})\subseteq\textnormal{span}\left(\{\mathsf{y}_{\mub_1},\dots,\mathsf{y}_{\mub_N}\}\right)\quad\forall\mub\in M.
    \end{equation*}
    In particular, POD-NN surrogates preserve homogeneous linear constraints. That is, if $D\mathsf{y}_{\mub}\equiv0$ for some linear operator $D$ and all $\mub\in M$, then $D\Phi_{\textnormal{POD}}\equiv0.$ 
    In particular, for any $\mub\in M$,
    \begin{equation*}
    \Phi_\textnormal{POD}(\mub)\in\textnormal{span}\left(\{\mathsf{y}_{\mub_i}\}_{i=1}^{N}\right)\subseteq\ker(D)\implies D\Phi_\textnormal{POD}(\mub)=0.
\end{equation*}
\end{remark}

\section{Learnable solvers that preserve linear and angular momentum} 
\label{sec:conservative solvers}

We now come to our main subject, which is the development of learnable solvers for parametrized PDEs with conservation properties. To this end, let $M\subset\mathbb{R}^{p}$ denote the parameter space. We consider elasticity problems of the form \eqref{eq:model-problem} where the operator $A$, the source term $\bodyforce$ and the boundary term $g_u$ depend on a vector of $p$ parameters, denoted as $\mub\in M$. 
We assume
that for each $\mub\in M$ the problem is well posed.
The FOM described in Section \ref{sub: FOM} defines a map
\begin{equation}
    \label{eq: FOM}
    M\ni\mub\mapsto (\sigma_{\mub},u_{\mub},r_{\mub})\in \Sigma_h\times U_h\times R_h.
\end{equation}
Note that $\Sigma_h\times U_h\times R_h$ is independent of $\mub$ because we adopt the discretization from \Cref{sub: FOM} for all parameter values. 

Learnable solvers aim to approximate the parameter-to-solution map \eqref{eq: FOM} by leveraging a set of pre-computed FOM solutions, called the \textit{training set}, obtained by sampling $N$ parameter configurations $\mub_1,\dots,\mub_N\in M$. Here, we focus on learnable solvers that employ neural networks, thus casting our problem in the general framework of \Cref{sec:dl preliminaries}. 

We focus on developing physically consistent neural network surrogates that ensure conservation of linear and angular momentum,
by leveraging
the spanning tree solver presented in \Cref{sec: spanning tree}. 
Recall that this solver is a linear operator $S_I:U_h'\times R_h'\to\Sigma_h$ that behaves as a right-inverse of the constraint operator $B$, meaning that $BS_If=f$ for all $f\in U'_h\times R_h'$.

Due to its efficiency, the spanning tree solver can readily be integrated within the computational pipeline of reduced order models without degrading their overall performance. This gives us two main advantages. First, as discussed in Section \ref{sec: spanning tree}, we can focus our attention on the stress field, i.e. on the map $\mub\to\sigma_{\mub}$. Using \eqref{eq: postprocessing}, we can leverage $S_I$ to post-process both the displacement $u$ and rotation $r$ from a given surrogate stress $\sigma$. Secondly, the operator $S_I$ allows us to construct subspaces of the kernel of $B$. For that, we define the operator
\begin{align*}
    S_0 \coloneqq I - S_I B.
\end{align*}
Now $S_0 : \Sigma_h \to \ker(B)$ because $BS_0 = B - IB = 0$. 

We can use this operator in two ways: either to split the stress tensor into a homogeneous term with respect to $B$ and a non-homogenoeus remainder, or to construct a corrector operator. The two approaches naturally give rise to two different strategies for reduced order modeling which we detail in \Cref{subsec: split,subsec: corrected}, respectively. An overview is provided in \Cref{tab:approaches}.

\renewcommand{\arraystretch}{1.5}
\begin{table}[ht]
    \centering
    \caption{Overview of the approaches considered for the approximation of the stress variable, complemented with their ansatz and corresponding formula (2nd and 3rd column, respectively).\\}
    \label{tab:approaches}
    \begin{tabular}{l|l|l}
    \hline\hline
    Non-linear regression  & $\stress_{\mub}\approx\tilde{\stress}_{\mub}$               & $=\Phi(\mub)$                                  \\\hline
    Split     & $\stress_{\mub}\approx\tilde{\stress}^0_{\mub}+\stress^{f}_{\mub}$ & $=\Phi_{0}(\mub)+S_{I}\fsource_{\mub}$\\\hline
    Corrected & $\stress_{\mub}\approx\hat{\stress}^{0}_{\mub}+\stress^{f}_{\mub}$ & $=\mathsf{V}_0\mathsf{V}_0^\top\left(\tilde{\stress}_{\mub}-\stress_{\mub}^f\right)+\stress_{\mub}^f$ \\
&                                         & $=\mathsf{V}_0\mathsf{V}_0^\top\left(\Phi(\mub)-S_{I}\fsource_{\mub}\right)+S_{I}\fsource_{\mub}$         \\
\hline\hline
    \end{tabular}
\end{table}

\subsection{The Split approach}
\label{subsec: split}
Following the strategy from \cite{boon2023deep}, we can exploit the operators $S_0$ and $S_I$ to decompose the stress tensor into a \emph{particular} solution, satisfying the conservation constratins, and a \emph{homogeneous} solution, lying in the kernel of $B$. In particular,
\begin{align*}
    \sigma 
    = \underbrace{\left(\sigma - S_I f\right)}_{\in \ker(B)} + S_I f\eqqcolon\sigma^0+\sigma^f.
\end{align*}
The key point here is that, in the parametric setting, the particular solution $\sigma^f=\sigma^f_{\mub}$ can be efficiently computed as
\begin{align*}
    \sigma^{f}_{\mub}
    =  S_I f_{\mub},
\end{align*}
In turn, we construct a neural network surrogate $\Phi_0:M\to\ker(B)$ to approximate the homogeneous component, $\Phi_0(\mub)\approx\sigma^0_{\mub}$, so that the stress approximation becomes:
\begin{align*}
    \sigma^{\text{split}}_{\mub} & \coloneqq \Phi_0(\mub) + S_I f_{\mub}.
\end{align*}
The only caveat is to construct $\Phi_0$ so that $\Phi_0(M)\subseteq\ker(B)$. Fortunately, this can be easily achieved via the POD-NN approach. As discussed in Remark \ref{remark:homogeneous}, networks constructed using the POD-NN strategy are guaranteed to preserve \emph{homogeneous} constraints, if the training set respects the homogeneous constraints.

In practice, we implement the \emph{Split approach} as follows. First, we exploit the FOM to compute a random collection of PDE solutions, $\{(\mub_i,\sigma_{\mub_i})\}_{i=1}^{N}$, which we hereon refer to as training \emph{snapshots}. Then, we leverage the operator $S_0$ to produce a training set of homogeneous solutions, 
\begin{equation*}
    \{(\mub_i, \sigma_{\mub_i}^0)\}_{i=1}^{N}=\{(\mub_i, S_0 \sigma_{\mub_i})\}_{i=1}^{N},
\end{equation*}

We then perform POD over these snapshots to produce a subspace $V_0\subseteq\ker(B)$ of dimension $n_{rb}$. Finally, we introduce a latent neural network $\phi : M \to \mathbb{R}^{n_{rb}}$ and define $\Phi_0$ as
\begin{align*}
    \mathsf{\Phi_0}(\mub) = \mathsf{V_0} \phi_0(\mub).
\end{align*}
Then, up to identifying the vector representation with its corresponding function representation, we have $\Phi_0 : M \to V_0 \subseteq\ker(B)$, as desired. 

\subsection{The Corrected approach}
\label{subsec: corrected}

A second strategy is to leverage the operators $S_0$ and $S_I$ to construct a \emph{corrector}. The idea is based on the following lemma.

\begin{lemma}
    \label{lemma: corrector}
    Let $V_0\subseteq S_0(\Sigma_h)=\ker B$ and let $\pi_0:\Sigma_h\to V_0$ be the orthogonal projection onto $V_0$ in the $\ell^2$ inner product. Define
    \begin{equation*}
        \epsilon \coloneqq \sup_{\mub\in M}\| (I - \pi_0) S_0 \sigma_{\mub} \|,
    \end{equation*}
    where $\sigma_{\mub} \in \Sigma_h$ is the FOM solution from \eqref{eq: FOM}, and $\| \cdot \|$ is the $\ell^2$-norm. Let the correction map $C:\Sigma_h\times M\to\Sigma_h$ be given by
    \begin{equation*}
        C(\tilde\sigma,\mub) \coloneqq \pi_0\left(\tilde\sigma-S_If_{\mub}\right)+S_If_{\mub}.
    \end{equation*}
    Then,
    for all $\tilde{\sigma}\in\Sigma_h$ and all $\mub\in M$ one has
    \begin{itemize}
    \item[i)] $BC(\tilde\sigma,\mub)=f_{\mub}$,
    \item[ii)] $\|\sigma_{\mub}-C(\tilde\sigma,\mub)\|\le\|\sigma_{\mub}-\tilde{\sigma}\|+\epsilon$.
\end{itemize}
\end{lemma}
\begin{proof}
    i) follows from the fact that $B\pi_0\equiv0$ and thus $BC(\tilde\sigma,\mub)=BS_If_{\mub}=f_{\mub}.$ For ii), we let $\sigma^0_{\mub} \coloneqq S_0 \sigma_{\mub} = \sigma_{\mub} - S_I f_{\mub}$. The orthogonality of $\pi_0$ then allows us to derive
    \begin{align*}
    \|\sigma_{\mub}-C(\tilde\sigma,\mub)\|^2 &= \|\sigma_{\mub}^0+S_If_{\mub}-C(\tilde\sigma,\mub)\|^2 \\&= \|\sigma_{\mub}^0-\pi_0(\tilde\sigma-S_If_{\mub})\|^2 \\&= 
    \|\sigma^{\mub}_0-\pi_0(\sigma_{\mub}^0)\|^2+\|\pi_0(\sigma_{\mub}^0)-\pi_0(\tilde\sigma-S_If_{\mub})\|^2
    \\&\le \epsilon^2 +
    \|\sigma_{\mub}^0-\tilde\sigma+S_If_{\mub}\|^2\\&=\epsilon^2+\|\sigma_{\mub}-\tilde\sigma\|^2. 
\end{align*}
Since $\sqrt{a+b}\le\sqrt{a}+\sqrt{b}$, the conclusion follows.
\end{proof}

\begin{remark}
Lemma \ref{lemma: corrector} can easily be adapted to a different norm, including $\|\cdot\|_{\Sigma_h}$ from \eqref{eqs:norms}.
\end{remark}

The operator $C$ in Lemma \ref{lemma: corrector} acts as a corrector in the sense that, given $\mub\in M$ and any $\tilde\sigma\in\Sigma_h$: it forces $C(\tilde\sigma, \mub)$ to satisfy the linear constraint associated to $\mub$ while simultaneously bringing $\tilde\sigma$ towards $\sigma_{\mub}$: see (i) and (ii), respectively. We emphasize that the subspace $V_0\subseteq\ker B$ needs to be chosen properly, so that $\epsilon$ is sufficiently small. In practice, this can be achieved as in \Cref{subsec: split}, by extracting a POD basis from a set of homogeneous snapshots $\{ S_0 \sigma_{\mub_i}\}_{i=1}^{N}$. 

We now have the ingredients necessary to construct the \emph{Corrected approach}. Given a dataset $\{\mub_i,\sigma_{\mub_i}\}_{i=1}^{N}\subseteq M\times\Sigma_h$, we first train a classical neural network model $\Phi$ such that $\Phi(\mub)\approx\sigma_{\mub}$. Then, we exploit the dataset and the operator $S_0$ to construct a subspace $V_0\subseteq\ker B$ of which we compute the corresponding projection $\pi_0.$ Then, following Lemma \ref{lemma: corrector}, we set
\begin{equation*}
    \sigma^{\text{cor}}_{\mub} \coloneqq \pi_0(\Phi(\mub)-S_If_{\mub})+S_If_{\mub},
\end{equation*}
as our final approximation for the stress field. Equivalently, using matrix-vector notation, $\sigma^{\text{cor}}_{\mub}=\mathsf{V}_0\mathsf{V}_0^\top(\Phi(\mub)-S_If_{\mub})+S_If_{\mub}$: see also \Cref{tab:approaches}. 

\begin{remark} \label{rem: Lipschitz}
    The Corrected approach follows a paradigm that is very close to the one adopted by the Split approach: the difference lies in how they approximate the homogeneous component of the stress field. It is worth mentioning that, from a practical point of view, this difference may prove significant. For instance, it has been observed that neural network models can efficiently approximate functions with small Lipschitz constant \cite{guhring2021approximation}. If we assume the parameter-to-solution map $\mub\to\sigma_{\mub}$ to be $L$-Lipschitz, then we may bound the Lipschitz constant of the homogeneous counterpart, $\mub\to\sigma^0_{\mub}$, as $\varrho_0 L$, where $\varrho_0$ denotes the spectral radius of $S_0$. The Corrected approach and the Split approach then entail learning a map whose Lipschitz constant is $L$ or $\varrho_0 L$, respectively. Consequently, if $\varrho_0>1$, we might expect the Corrected approach to perform better.
\end{remark}

\section{Numerical experiments}\label{sec:experiments}

We devote this section to the empirical assessment of the proposed approaches. To this end, we consider three test cases of increasing complexity: a footing problem in 2D with four parameters (\Cref{subsec:ex1}), a cantilever problem in 3D, featuring three parameters (\Cref{subsec:ex2}), and a non-linear problem in 2D based on the Hencky-von Mises model, where four parameters describe the non-linear constitutive relations of the Lamé coefficients.

In order to verify the advantages of the proposed approaches over standard neural network regression, our analysis will include two benchmark models in which the stress tensor is approximated using a purely data-driven approach, $\sigma_{\mub}\approx\Phi(\mub)$, whereas the displacement $u$ and the rotation $r$ are obtained via the post-processing of \eqref{eq: postprocessing}. Specifically, we consider the following benchmark models:
\begin{itemize}
    \item \textit{Black-box}. In this case, $\Phi$ is a classical feed forward neural network, designed and trained as in \Cref{sub: regression};
    \item \textit{POD-NN}. In this case, the neural network model approximating the stress variable is constructed using the POD-NN strategy. We emphasize that, despite its black-box nature, this approach is consistent with the conservation of \emph{angular} momentum \eqref{eq: ang momentum}. This is because POD-based surrogates are known to preserve homogeneous constraints, as noted in \Cref{remark:homogeneous}. On the other hand, there is no guarantee on the conservation of \emph{linear} momentum.
\end{itemize}

To ensure a fair and systematic comparison, all approaches utilize neural networks with the same architecture, with structural changes only allowed between test cases. Details concerning the architecture are given in \Cref{appendix: architectures}. For each test case, we train the models on $N_{\text{train}} = 150$ randomly generated snapshots and assess their performance using an independent test set of $N_{\text{test}}=50$ instances.
Quantitatively, we measure the quality of the approximation using the $\Sigma_h$-mean relative error ($\Sigma_h\text{-}\MRE$), computed as
\begin{gather}
    \label{eq: l2mre}
    \sigmaMRE = 
    \frac{1}{N_{\text{test}}}\sum_{i=1}^{N_{\text{test}}}\frac{\|\sigma_{\mub_i^{\text{test}}}-\tilde{\sigma}_{\mub_i^{\text{test}}}\|_{\Sigma_h}}{\|\sigma_{\mub_i^{\text{test}}}\|_{\Sigma_h}}.
\end{gather}
We use analogous metrics for the displacement $u$ and the rotation $r$, respectively denoted as $\uMRE$ and $\rMRE$, cf. Eq. \eqref{eqs:norms}. Here, $\{\mub_i^{\text{test}}\}_{i=1}^{N_{\text{test}}}$ are the parameter instances in the test set, whereas $\sigma_{\mub_i^{\text{test}}}$ and $\tilde{\sigma}_{\mub_i^{\text{test}}}$ are the FOM reference and the ROM approximation, respectively.

To evaluate the fulfillment of the conservation laws, we also compute the average constraint violation ($\ACV$), given by
\begin{gather}
    \label{eq: constraint violation}
    \ACV = 
    \frac{1}{N_{\text{test}}}\sum_{i=1}^{N_{\text{test}}}\|B\tilde{\sigma}_{\mub_i^{\text{test}}}-f_{\mub_i^{\text{test}}}\|_{\ell^{\infty}}.
\end{gather}

All test cases were coded in Python 3 using the libraries PorePy \cite{Keilegavlen2020}, PyGeoN \cite{pygeon}, DLROMs \cite{franco2024dlroms}, and a dedicated module available at \url{github.com/compgeo-mox/conservative_ml_elasticity}. All computational grids are made of shape-regular simplices, generated by Gmsh \cite{Geuzaine2009}. 

\subsection{Linearized elasticity in 2D} \label{subsec:ex1}

In this first test case, we consider the classical footing problem with a body force in two dimensions. The domain $\Omega$ is given by the unit square. At the bottom boundary, a no-displacement condition is applied, $u=0$, while homogeneous normal stress $\nu \cdot \sigma = 0$ is imposed on the left and right sides. On the top boundary, we impose a parameter dependent, downward force 
\begin{gather*}
    {g_u}=10^{-3}\cdot[0,-g_y]^\top
    \quad \text{with} \quad 0.5\le f_y\le2.
\end{gather*}
To make the constraint associated to the linear momentum equation 
more challenging, 
we consider a parametric body force equal to 
\begin{gather*}
    {\bodyforce}=10^{-2}\cdot[0,-f_y]^\top
    \quad \text{with} \quad 0.5\le g_y \le2.
\end{gather*}
Finally, we set the parameter dependent Lamé coefficients in the ranges 
\begin{gather*}
    0\le\mu\le2
    \quad \text{and} \quad  0.1\le\lambda\le2.   
\end{gather*}
The parameter space is thus four dimensional, with $\mub=[g_y,f_y,\mu,\lambda].$
For the spatial discretization, we use a computational grid of 242 triangles with mesh size $h=0.1$.

The results are presented in \Cref{tab:case-1} and \Cref{fig:boxplot-case-1,fig:example-case-1}.
Overall, the approaches provide a satisfactory approximation of the stress field when evaluated using the $\Sigma_h$-metric. In particular, \Cref{tab:case-1} shows that all methods report an average $\Sigma_h$ error below $2\%$. However, these inaccuracies propagate to the variables $u$ and $r$, with the rotation field being the most sensitive.

Notably, the approximations provided by the Black-box and POD-NN approaches show a significant disagreement with the conservation laws, with $\ACV$ values 9 orders of magnitude larger than those reported by our proposed approaches. In contrast, both the Split and the Corrected approach manage to fulfill the conservation constraints up to machine precision, as expected by the theory.

\renewcommand{\arraystretch}{1.5} 
\begin{table}[ht]
    \centering
    \caption{Results for the case study in Section \ref{subsec:ex1}. We report the $\MRE$s for stress, $\sigma$, displacement, $u$, and rotation, $r$, for the three ROM approaches. Each MRE is computed according to the natural norm associated to the underlying finite element space, see, e.g., Eq. \eqref{eq: l2mre}. The last column reports the average constraint violation $\ACV$ computed as in Eq. \eqref{eq: constraint violation}.\\}
    
    \label{tab:case-1}
    \begin{tabular}{lcccc}
        \hline\hline        \textbf{ROM}&$\sigmaMRE\;\;[\sigma]$&$\uMRE\;\;[u]$&$\rMRE\;\;[r]$&$\ACV$\\
        \hline
        Black-box   & 9.86e-03    &    1.28e-02   &     1.32e-01   &     3.03e-03 \\
        POD-NN     & 6.39e-03                    & 9.12e-03               & 7.38e-02               & 3.30e-05                      \\
        \rowcolor{Gray}Corrected    & 5.96e-03                    & 9.25e-03               & 7.34e-02               & 1.14e-14                      \\
        Split        & 1.37e-02                    & 7.76e-02               & 7.34e-01               & 1.13e-14                      \\
        \hline\hline
    \end{tabular}
\end{table}

\begin{figure}[ht]
    \centering
    \includesvg[width=0.8\textwidth]{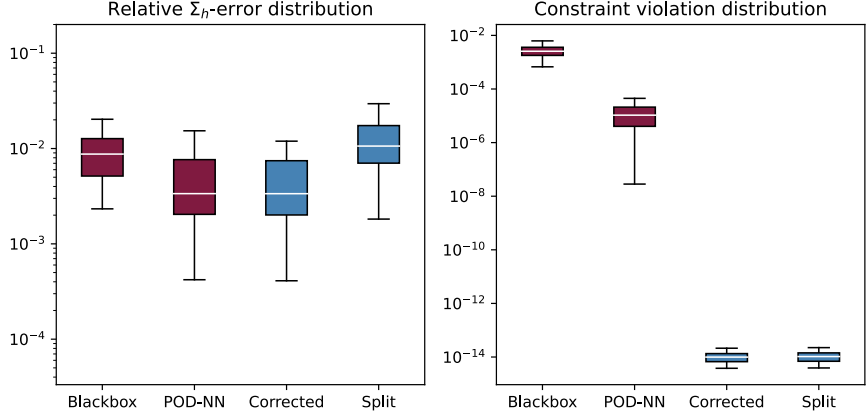}
    \caption{Errors distribution in the stress variable for the case study of \Cref{subsec:ex1}. $\Sigma_h$-relative errors (left) are obtained by the summands in \eqref{eq: l2mre}. Similarly, measurements for the constraint violation (right) refer to \eqref{eq: constraint violation}.}
    \label{fig:boxplot-case-1}
\end{figure}

\begin{figure}[ht]
    \centering
    \includegraphics[width=0.28\linewidth]{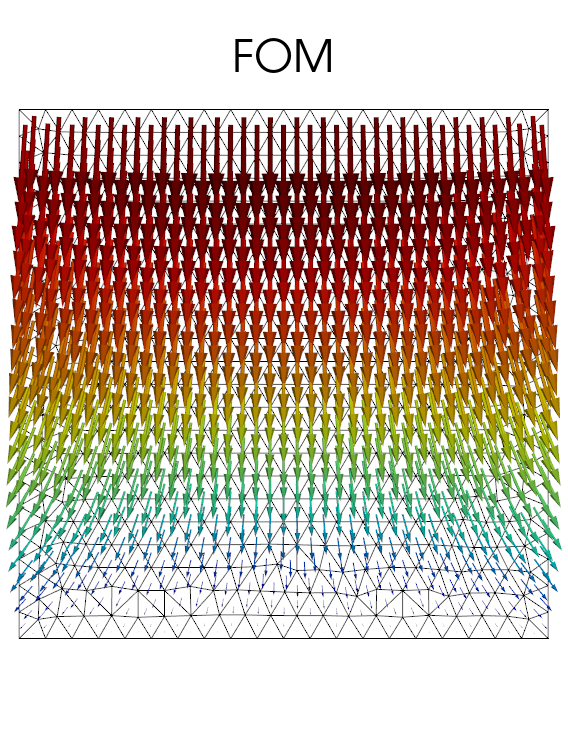}
    \includegraphics[width=0.28\linewidth]{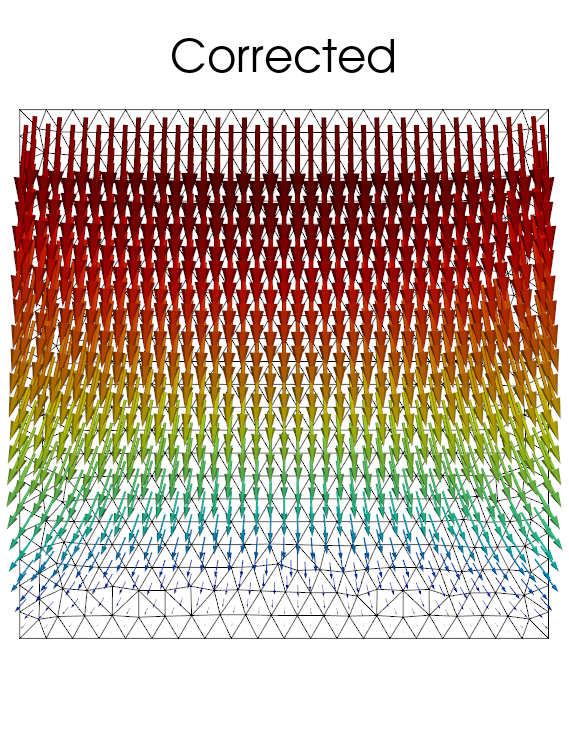}
    \includegraphics[width=0.28\linewidth]{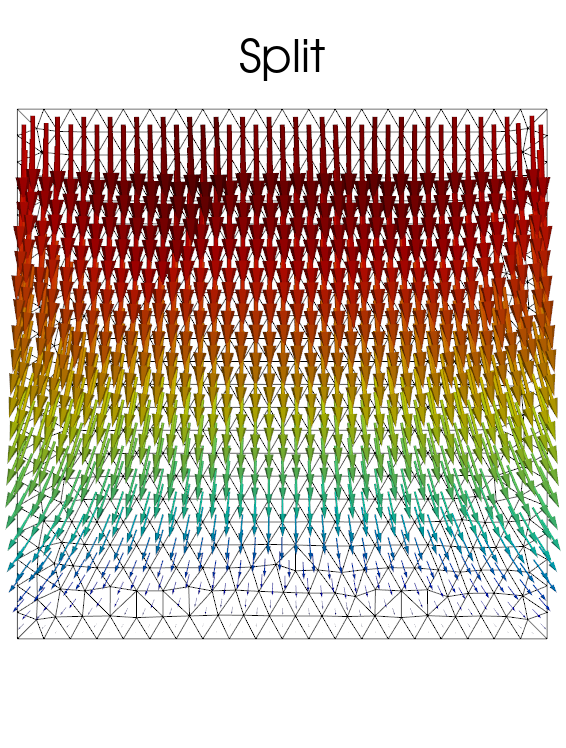}
    \includegraphics[width=0.28\linewidth]{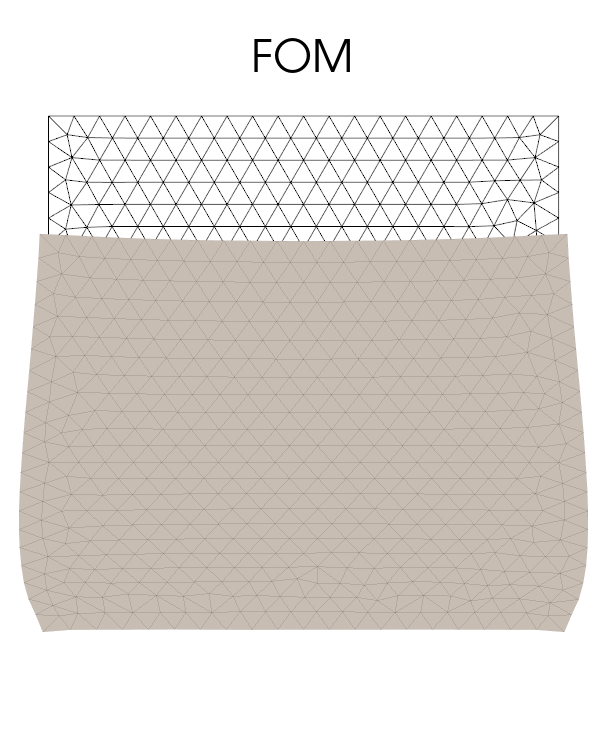}
    \includegraphics[width=0.28\linewidth]{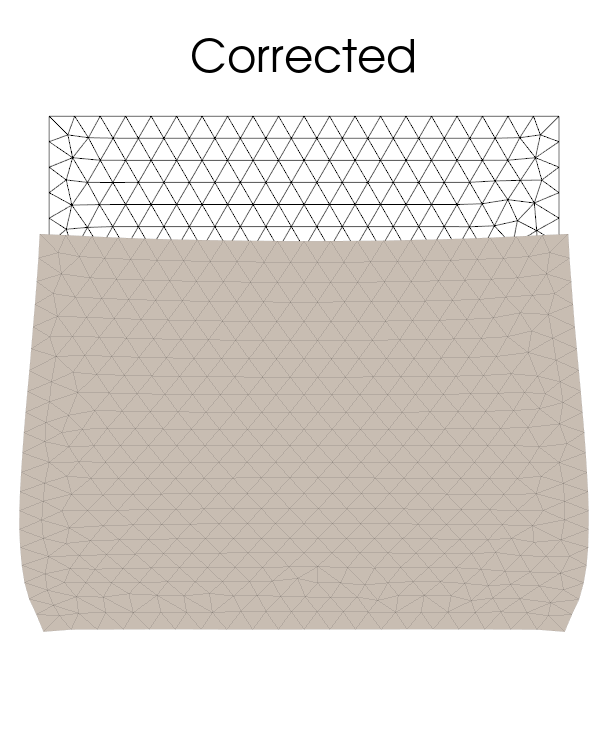}
    \includegraphics[width=0.28\linewidth]{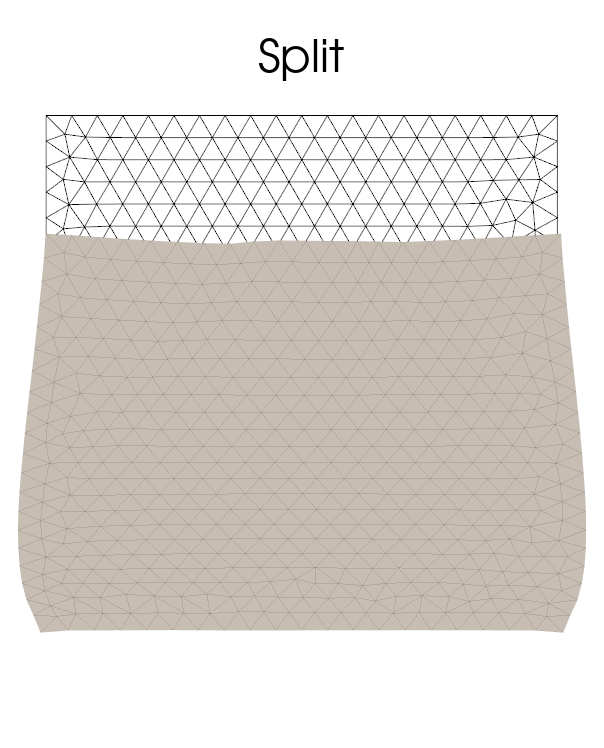}
    \caption{Comparison between FOM and conservative solvers for the footing problem of Section \ref{subsec:ex1}. Parameter values have been selected at random from the test set: $g_y=1.167e-03 $, $f_y=7.800e-03$, $\mu=1.055$, $\lambda=1.593$. Top: quiver plots of the displacement field $u$. Bottom: original domain (grid) and exaggerated, deformed domain according to the computed displacement (grey; exaggeration factor: 150).}
    \label{fig:example-case-1}
\end{figure}

As seen in \Cref{fig:boxplot-case-1}, these consideration not only hold in the average sense, but also from a global perspective.
There, we can also appreciate that the POD-NN and the Corrected approaches produce comparable error distributions, while the Split and Black-box methods show a less favorable performance.
To summarize, the Corrected approach shows the best performance for this case, as it manages to ensure physical consistency while also achieving state-of-the-art performance with respect to $\Sigma_h$ and $L^{2}$ metrics.

The Split approach, instead, guarantees momentum conservation but produces larger errors, which result in a poor approximation of the displacement $u$ and the rotation $r$. A visual inspection of the displacement fields, cf. \Cref{fig:example-case-1}, highlights this phenomenon. While the approximation produced by the Corrected approach is satisfactory, we notice that the Split approach produces some geometrical artifacts that seem related to the structure of the spanning tree solver.

\subsection{Linearized elasticity in 3D}\label{subsec:ex2}

As our second test case, we consider a cantilever fixed at one end and subjected to a body force. The computational domain is 
given by $\Omega = (0, 2)\times (0, 0.5) \times (0, 0.5)$. On the left side, i.e. for $x =0$, we set $u=0$, while on the remaining boundary 
we impose zero traction, $\nu \cdot \sigma = 0$. The body force is parametric and given by
\begin{gather*}
    {\bodyforce}=10^{-2}\cdot[0,0,-f_z]^\top
    \quad \text{with} \quad 0.5\le f_z\le2.    
\end{gather*}
In this case, the Lamé coefficients are parametrized as
\begin{gather*}
    0\le\mu\le2
    \quad \text{and} \quad
    0.1\le\lambda\le2.   
\end{gather*}
Consequently, this case study features three parameters, $\mub=[f_z,\mu,\lambda]$. To construct the FOM we use a mesh of stepsize $h=0.2$, resulting in a computational grid with 412 tetrahedra.

The numerical results are reported in \Cref{tab:case-2} and Figures \ref{fig:boxplot-case-2}-\ref{fig:example-case-2}.
As in the previous test case, we observe that in terms of $\Sigma_h$ and $L^{2}$ accuracy, the Corrected approach performs comparably to both POD-NN and naive Black-box regression, while the Split algorithm produces errors that are an order of magnitude higher than the others. Again, the Corrected and the Split approaches fulfill the physical constraints exactly, by construction. 
In this respect, POD-NN performs slightly better than Black-box regression. Indeed, \Cref{fig:boxplot-case-2} shows that POD-NN yields constraint residuals below $10^{-4}$, whereas the Black-box performs poorly with a residual above $10^{-3}$. This effect is most likely caused by the inherent ability of POD-NN to account for the conservation of angular momentum (cf. \Cref{remark:homogeneous}). The conservation of linear momentum, however, is not guaranteed, resulting in an $\ACV$ of $1.82\cdot10^{-5}$ for POD-NN, as shown in \Cref{tab:case-2}.

\begin{table}[ht]
    \centering%
    \caption{Results for the test case in Section \ref{subsec:ex2}. Table entries read as in \Cref{tab:case-1}.\\}\vspace{0.2cm}
    \label{tab:case-2}
    \begin{tabular}{lcccc}
        \hline\hline        \textbf{ROM}&$\sigmaMRE\;\;[\sigma]$&$\uMRE\;\;[u]$&$\rMRE\;\;[r]$&$\ACV$\\
        \hline
        Black-box  & 6.09e-03    &    3.20e-03   &     3.30e-03    &    4.32e-03\\
        POD-NN     & 5.17e-03                    & 4.05e-03               & 4.05e-03               & 1.82e-05                      \\
        \rowcolor{Gray}Corrected    & 3.52e-03                    & 3.21e-03               & 3.05e-03               & 1.64e-13                      \\
        Split        & 1.49e-02                    & 3.10e-02               & 2.84e-02               & 1.69e-13                      \\
        \hline\hline
    \end{tabular}
\end{table}

\begin{figure}[ht]
    \centering
    \includesvg[width=0.8\textwidth]{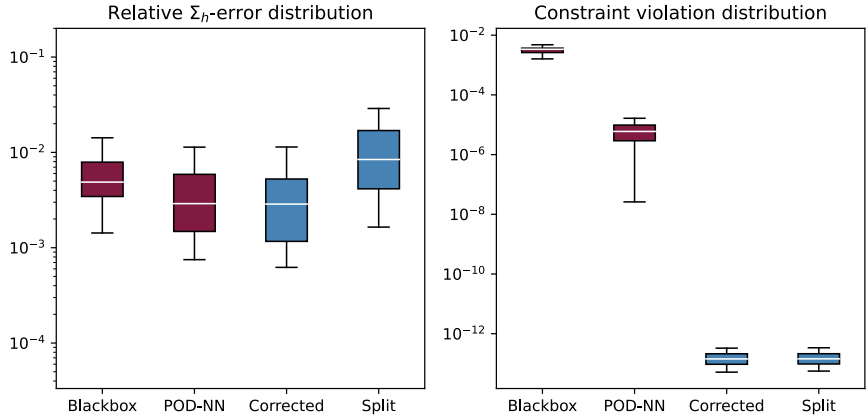}
    \caption{$\Sigma_h$-relative errors in the stress and constraint violation for the second test case, \Cref{subsec:ex2}. Panels read as in \Cref{fig:boxplot-case-1}.}
    \label{fig:boxplot-case-2}
\end{figure}

\begin{figure}[ht]
    \centering
    \includegraphics[width=0.32\linewidth]{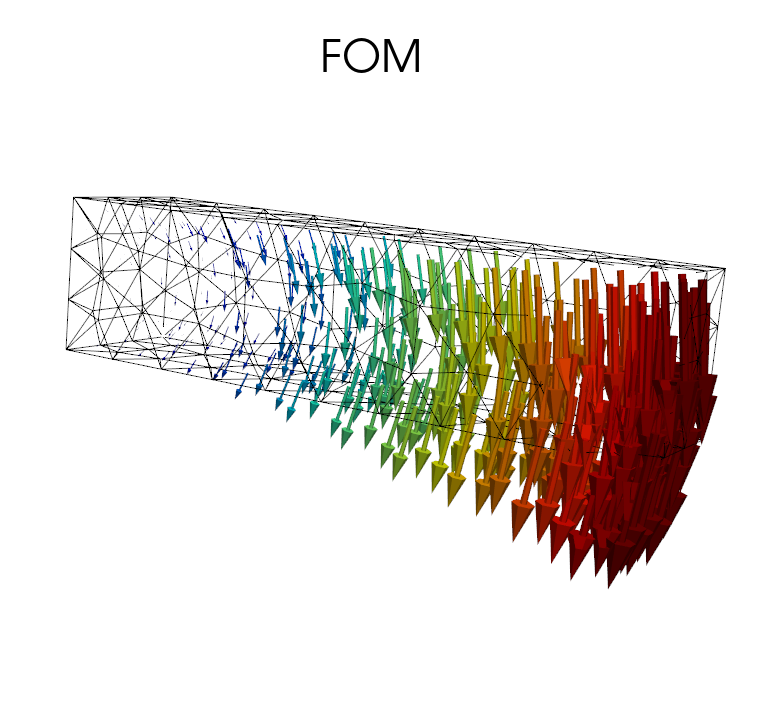}
    \includegraphics[width=0.32\linewidth]{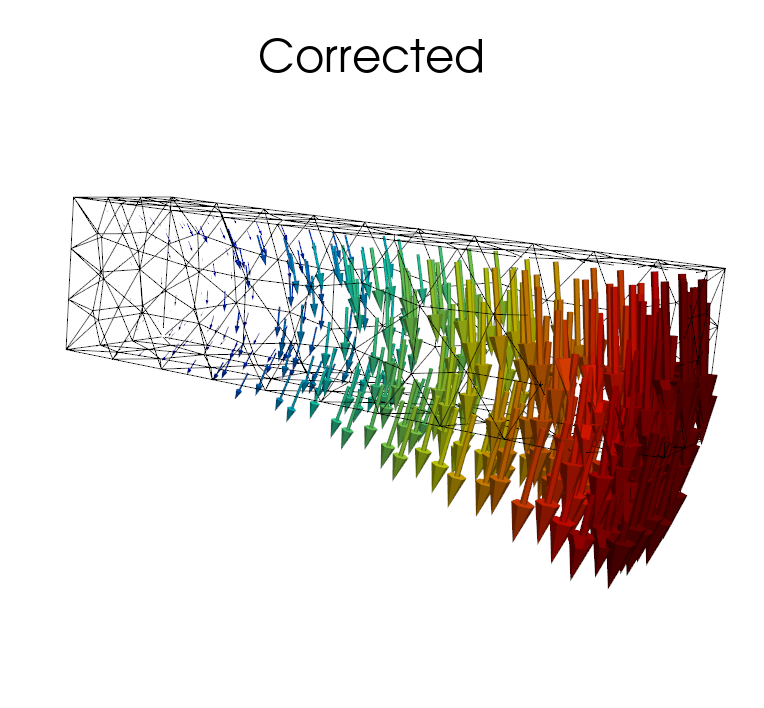}
    \includegraphics[width=0.32\linewidth]{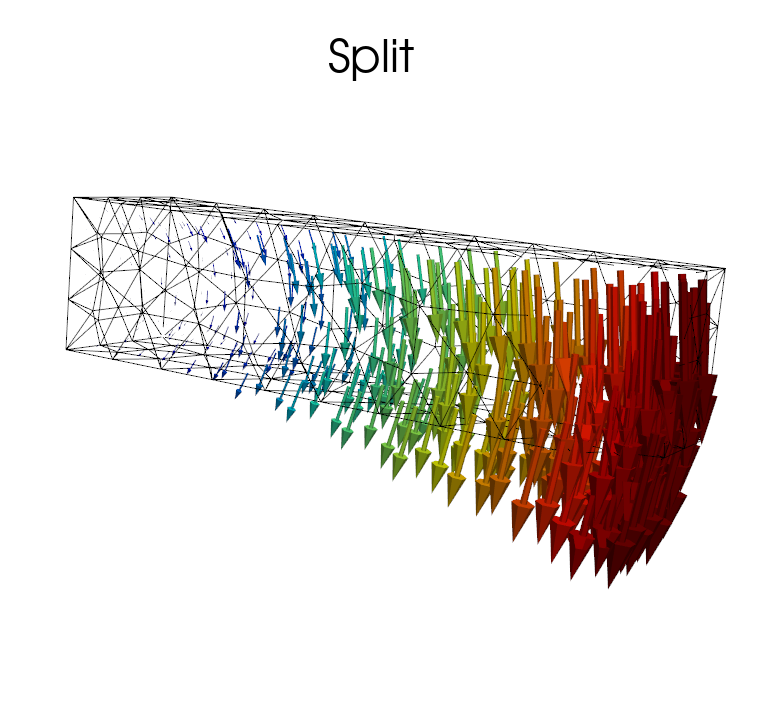}
    \includegraphics[width=0.32\linewidth]{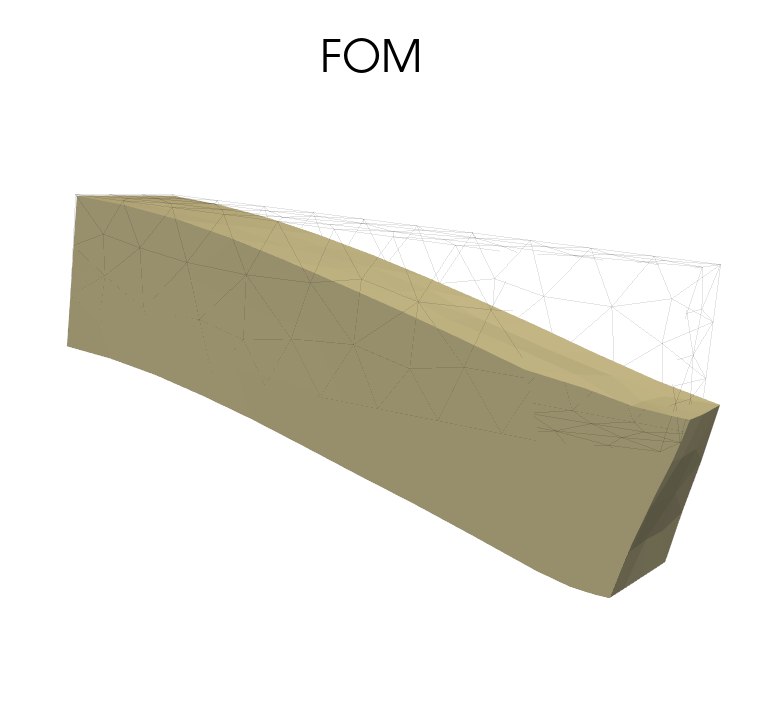}
    \includegraphics[width=0.32\linewidth]{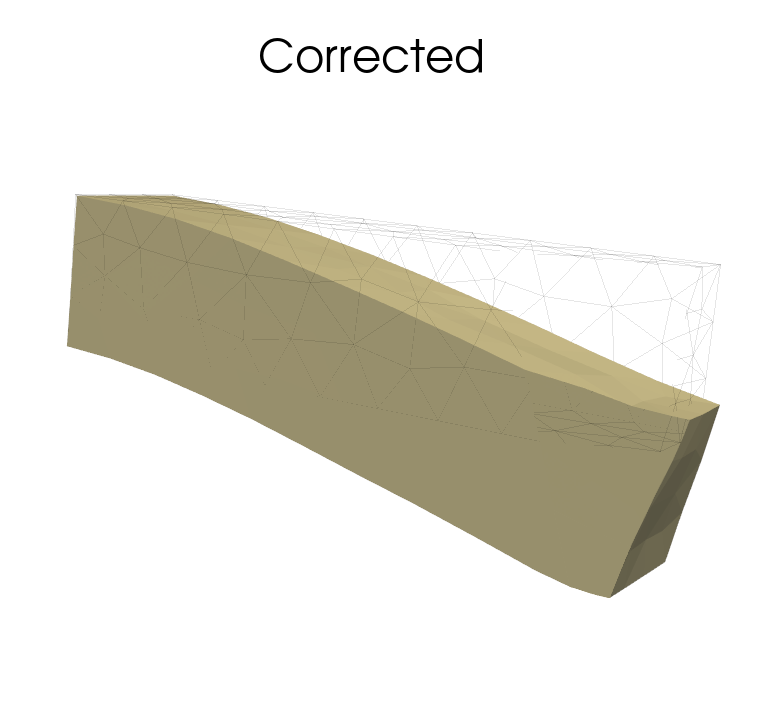}
    \includegraphics[width=0.32\linewidth]{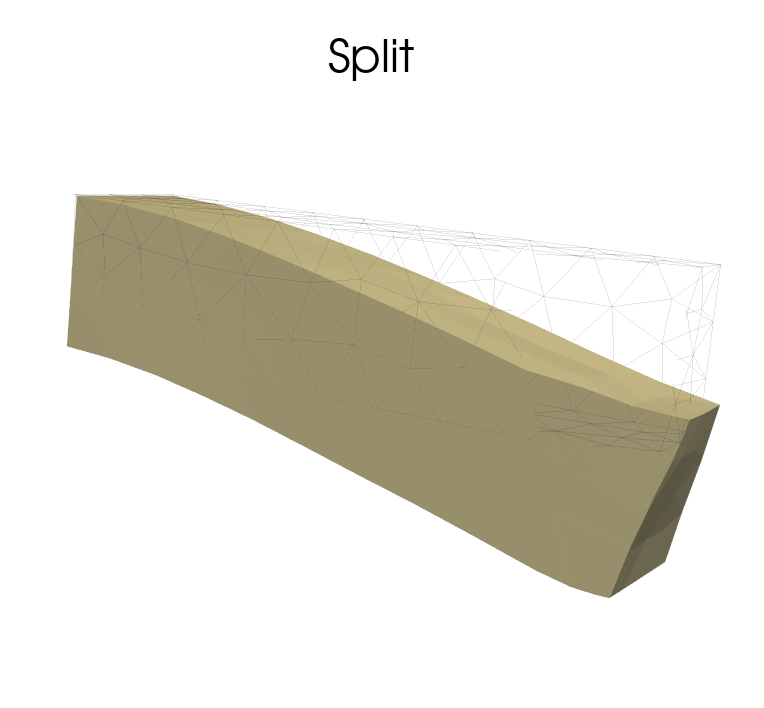}
    \caption{Comparison between FOM and conservative solvers for the cantilever problem of Section \ref{subsec:ex2}. Parameter values have been selected at random from the test set: $\mu= 1.826$, $\lambda= 0.590$, $f_z= 0.0182$. Top: quiver plots of the displacement field $u$. Bottom: original domain (grid) and corresponding 
    deformed domain (grey). The quiver plots are directly generated in PyVista \cite{sullivan2019pyvista} using the displacement as cell data. The generation of warp plots, on the other hand, relies on an interpolation procedure (intrinsic to PyVista) which maps cell data to point data. This may generate visual artifacts when the mesh is coarse, as apparent at the right end of the domain.}
    \label{fig:example-case-2}
\end{figure}

We note that the error propagation during the post-processing of phase is milder in this case. In particular, we notice that all methods manage to produce reasonable approximations of the displacement field. In fact, the FOM and the ROM simulations are nearly indistinguishable in \Cref{fig:example-case-2}.

\subsection{Non-linear elasticity in 2D}\label{subsec:ex3}

In this last test case, we consider a non-linear model of Hencky-von Mises type, cf. \Cref{ex: hencky-von mises}. Let $\Omega=(0,1)^2$ be the unit square, on which we generate a grid of 944 triangles with a mesh size of $h=0.05$. We consider a parametrized variant of the model in \cite{gatica2013priori}, where the Lamé parameters depend on the norm of the deviatoric strain $\zeta = \| \dev \varepsilon u \|$. Specifically, we consider the following expressions for $\mu$ and $\lambda$
\begin{align*}
    \lambda(\zeta) & \coloneqq \alpha \left(1 - \frac12 \mu(\zeta)\right), &
    \mu(\zeta)     & \coloneqq 1 + (1 + \zeta^2)^{\frac{\beta - 2}2},
\end{align*}
for given parameters $\alpha\ge0$ and $\beta\le 2$. Note that $1\le \mu(\zeta)\le 2$ and $0\le\lambda(\zeta)\le\alpha/2$, resulting in a physically relevant range for the Lamé parameters.

We impose the displacement as $g_u$ on the entire boundary and include a body force $\bodyforce$, given by
\begin{align*}
    g_u(x, y)     & = \frac{\gamma}{10}\begin{bmatrix}
        x(1 - x) \\
        y(1 - y)
    \end{bmatrix},        &
    \bodyforce(x, y)       & = \delta
    \begin{bmatrix}
        \left( 4y - 1 \right)
        \left( 4y - 3 \right) \\
        \left( 4x - 1 \right)
        \left( 4x - 3 \right)
    \end{bmatrix},
\end{align*}
with $\gamma$ and $\delta$ model parameters. We allow the four parameters $\mub=[\alpha,\beta,\gamma,\delta]$ to attain values in the following ranges
\begin{align*}
    1  & \le \alpha \le 2, &
    0  & \le \beta \le 2,  &
    -1 & \le \gamma \le 1, &
    -1 & \le \delta \le 1.
\end{align*}

The numerical results are reported in \Cref{tab:case-3} and \Cref{fig:boxplot-case-3,fig:example-case-3}. Again, the Corrected approach outperforms both POD-NN and Black-box. The Split approach exhibits the worst performance in all metrics. Both benchmark models violate the physical constraints significantly, with $\ACV$ values above $10^{-3}$, while the proposed approaches manage to conserve linear and angular momentum up to machine precision. As shown in \Cref{fig:boxplot-case-3}, these observations apply consistently across the entire parameter space.

\begin{table}[ht]
    \centering
    \caption{Results for the test case in Section \ref{subsec:ex3}. Table entries read as in \Cref{tab:case-1}.\\}\vspace{0.2cm}
    \label{tab:case-3}
    \begin{tabular}{lcccc}
        \hline\hline
        \textbf{ROM}&$\sigmaMRE\;\;[\sigma]$&$\uMRE\;\;[u]$&$\rMRE\;\;[r]$&$\ACV$\\
        \hline
        Black-box   & 6.23e-03   &     6.52e-03    &    6.18e-03   &     1.17e-02 \\
        POD-NN     & 7.46e-04                    & 7.87e-04               & 8.86e-04               & 1.16e-03                      \\
        \rowcolor{Gray}Corrected    & 5.91e-04                    & 7.07e-04               & 7.42e-04               & 1.37e-13                      \\
        Split        & 3.83e-03                    & 1.06e-01               & 1.62e-01               & 1.39e-13                      \\
        \hline\hline
    \end{tabular}
\end{table}

\begin{figure}[ht]
    \centering
    \includesvg[width=0.8\textwidth]{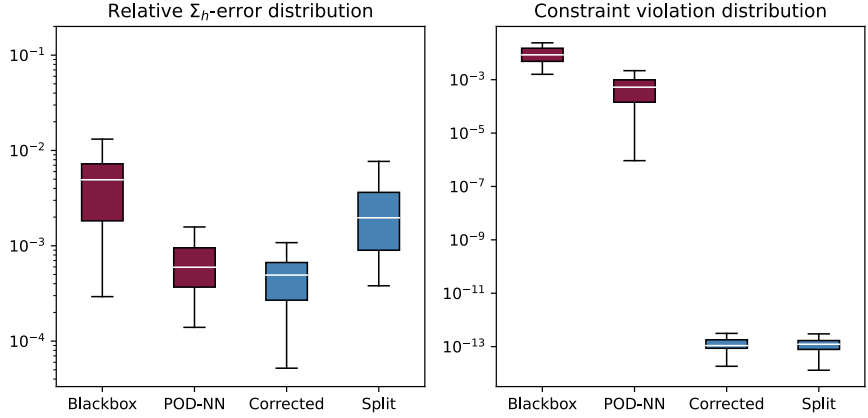}
    \caption{$\Sigma_h$-relative errors for the stress and constraint violation for the test case of \Cref{subsec:ex3}. Panels read as in \Cref{fig:boxplot-case-1}.}
    \label{fig:boxplot-case-3}
\end{figure}

\begin{figure}[tb]
    \centering
    \includegraphics[width=0.28\linewidth]{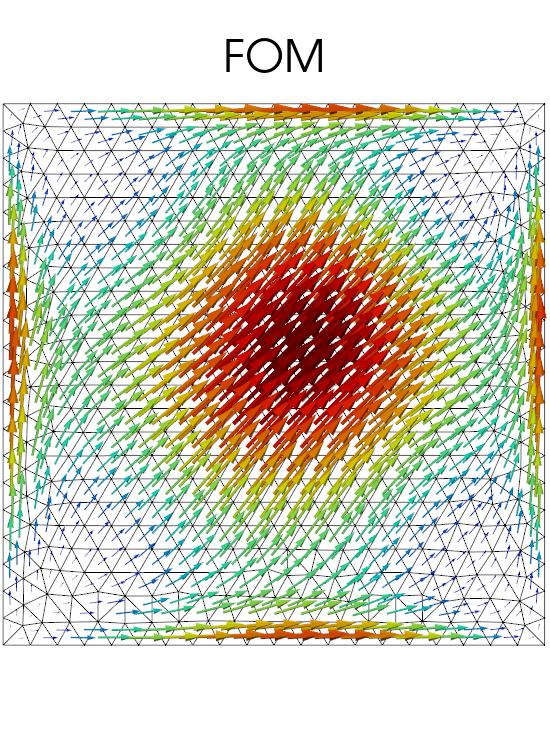}
    \includegraphics[width=0.28\linewidth]{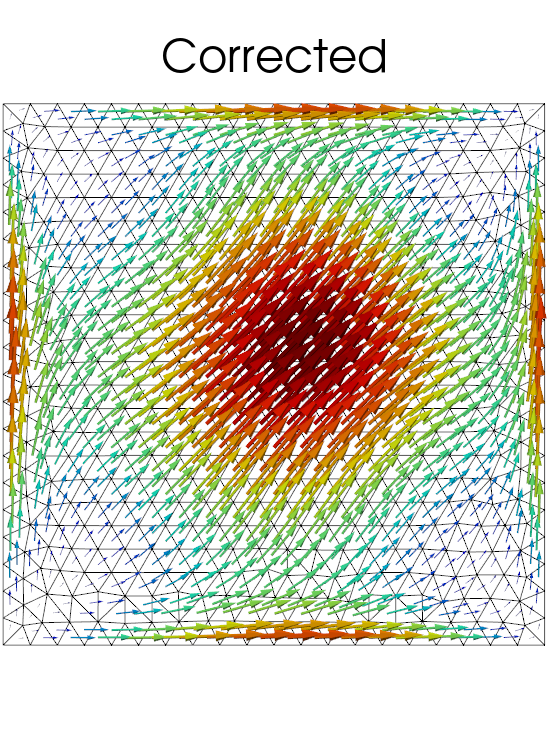}
    \includegraphics[width=0.28\linewidth]{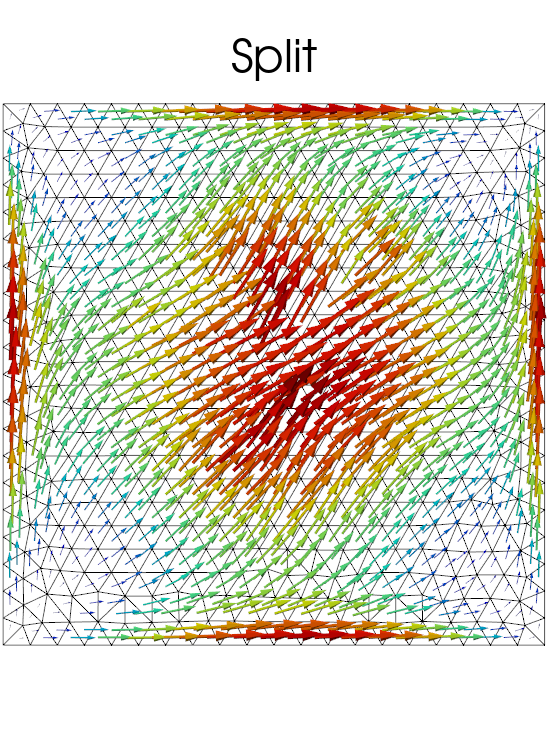}
    \includegraphics[width=0.28\linewidth]{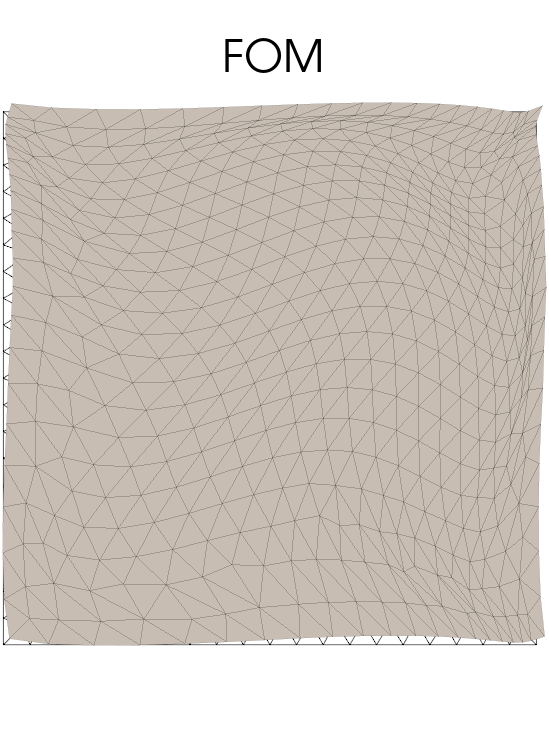}
    \includegraphics[width=0.28\linewidth]{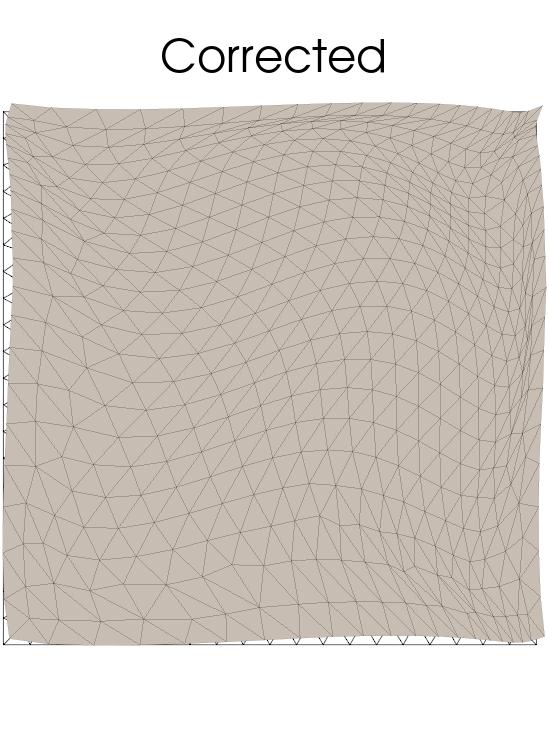}
    \includegraphics[width=0.28\linewidth]{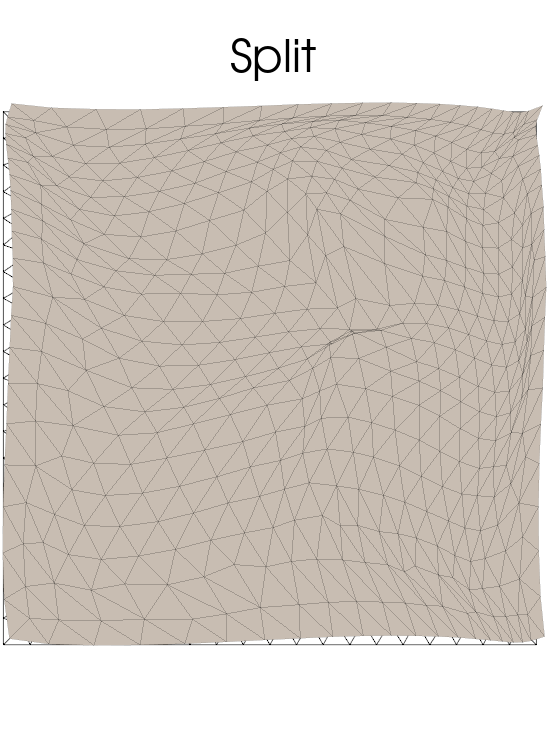}
    \caption{Comparison between FOM and conservative solvers for the Hencky-von Mises model of Section \ref{subsec:ex3}. Parameter values have been selected at random from the test set: $\alpha = 1.311$, $\beta=1.601$, $\gamma = 0.540$, $\delta = -0.646$. Top: quiver plots of the displacement field $u$. Bottom: the deformed domain (exaggeration factor: 20).}
    \label{fig:example-case-3}
\end{figure}

\begin{table}[ht!]
    \centering
    \caption{Comparison of computational times for the Hencky-von Mises problem of Section \ref{subsec:ex3}. Results limited to the best ROM according to \Cref{tab:case-3}. Setup time refers to the time required to initialize the solver (thus not affecting subsequent queries). Computing $\sigma$ is the time required to compute the stress field. Getting $(u,r)$ refers to the time needed to produce a full solution.\\}
    \label{tab:times}
    \begin{tabular}{llll}
        \hline\hline
        & Setup time & Computing $\sigma$ & Computing $(u,r)$ \\
        \hline
        \textbf{FOM}   & 6.95 s &  54.62 s  &  0 s  \\
        \textbf{ROM (Corrected)}   & 7.32 s &  0.0089 s  &  11.46 s  \\
        \hline\hline
    \end{tabular}
\end{table}

Among classical methods, POD-NN is the better alternative, but it fails to provide physical consistency. Our proposed approaches, instead, are capable of guaranteeing momentum conservation, in line with Sections \ref{subsec:ex1}-\ref{subsec:ex2}. However, only the Corrected method succeeds in doing so without introducing numerical artifacts. The Split approach, on the other hand, yields unsatisfactory results in all variables.

This is particularly evident in \Cref{fig:example-case-3}, where we illustrate the deformed domain based on the computed displacements. While the Corrected solution is indistinguishable from the FOM, we observe large deviations for the Split approach, especially in the center of the domain. Once again, this suggests a high sensitivity of the Split approach with respect to the spanning tree, as alluded to in \Cref{rem: Lipschitz}. 
In this case, we used the multi-root spanning tree illustrated in \Cref{fig:spt}(right), which results in multiple independent trees reaching towards the center of $\Omega$. This structure is clearly visible in the approximation proposed by the Split approach, suggesting a strong correlation between the two.

We conclude with a final observation concerning the computational cost\footnote{Clock times were recorded using a personal laptop: MSI Prestige 15 with an 11th-gen 3GHz Intel i7 processor, 16GB of RAM, and an NVIDIA GeForce GTX 1650 Max-Q GPU (4GB). Consequently, the reported values are inherently dependent on the specific hardware and implementation, but they still serve as a useful basis for performance comparison.}. 
Due to the non-linear nature of the problem, each simulation of the FOM requires roughly 1 minute (see \Cref{tab:times}). Approximately $7$ seconds are required for setting up the model, while solving the system takes $\sim 54$ seconds, for each parameter instance $\mub$. The FOM produces the three fields $(\sigma_h,u_h,r_h)$ simultaneously since the triplet is considered as a single unknown in the product space $\Sigma_h\times U_h\times R_h$. 

In contrast, the proposed deep learning based ROMs are significantly faster, as shown in \Cref{tab:times}. Although the ROM requires the same amount of time to be initialized, roughly $7$ seconds, generating solutions for the stress field only takes a few milliseconds. 
In particular, once the FOM and the ROM have been set up, the latter is $\times 6000$ faster in computing the stress field. However, the post-processing of the displacement $u$ and rotation $r$ requires an additional time of around $\sim 11$ seconds. This additional cost is due to the additional assembly of several operators at the FOM level for given $\sigma_h$. Nevertheless, this burden is mitigated by the efficiency of the spanning tree solver.

Thus, the proposed ROM can provide varying levels of speed-up, depending on whether the entire solution is required or if only the stress field is needed. Specifically, for applications where determining the stress field $\sigma$ is of interest, e.g. in evaluating fault stability or assessing structural integrity of the medium, our proposed approach provides an attractive advantage.

\newpage\section{Conclusions} \label{sec:conclusion}

In this work, we have introduced two new reduced order models for elasticity problems that guarantee conservation of linear and angular momentum. The analysis was conducted within the context of data-driven ROMs, where, for problems of this type, methods such as POD-NN constitute the state-of-the-art. Therefore, our effort can be understood as an attempt to bring data-driven surrogates closer to classical physics-based ROMs, such as POD-Galerkin and other Reduced Basis methods.

These two approaches are based on a suitable decomposition of the stress into an homogeneous and a particular parts, where the latter does not depends on the parameters and is the same for both approaches. The actual construction of the homogeneous solution is method specific. The computation of the particular solution is based on a spanning tree strategy to appropriately select suitable degrees of freedom, in order to obtain a (block) triangular system and thus making the computation of the particular solution inexpensive. On the other side, the homogeneous solutions are built based on suitable trained reduced models. 
In the online phase, the resulting approaches are thus effective and computationally efficient in providing approximations of the stress field fulfilling the linear constraints exactly.

We assessed the capabilities of the proposed approaches by comparing their performances against POD-NN (a state-of-the-art approach within the class of data-driven ROMs) and naive neural network regression, which, given our proposal, constitute a natural benchmark. The latter, in fact, can be very accurate when considering crude $L^{2}$ metrics\footnote{Recall that $\|\sigma\|_{L^2(\Omega)}\le\|\sigma\|_{\Sigma_h}$ and $\|u\|_{U_h}=\|u\|_{L^2(\Omega)},$ $\|r\|_{R_h}=\|r\|_{L^2(\Omega)}.$}, but they are not guaranteed to yield physical results.

All our numerical experiments suggest the same conclusion: while black-box methods based on neural networks can achieve small relative errors in the $L^{2}$ sense, their approximations consistently violate the momentum equations, with maximum cell-wise residuals in the order of $10^{-2}-10^{-5}$. Our proposed approaches, instead, are always in exact agreement with the conservation laws, yielding cell-wise residuals close to machine precision. However, of the two proposed strategies (Corrected and Split), only the Corrected approach managed to achieve this physical consistency while simultaneously delivering $L^{2}$-accurate solutions, thus consistently outperforming all benchmark models. The Split method, on the other hand, was not as convincing: its approximations of the stress field were often affected by numerical artifacts and, during the post-processing phase, these inaccuracies tended to propagate over both the displacement $u$ and the rotation $r$.

In conclusion, we presented two possible approaches for the construction of physically consistent neural network solvers compatible, one of which was capable of achieving state-of-the-art performance while ensuring the conservation of both linear and angular momentum. The approach also showed remarkable speed ups with respect to the FOM, especially for what concerns the computation of the stress field. This can be particularly interesting for applications in geomechanics and civil engineering involving, for instance, the evaluation of structural integrity, fault stability or crack propagation.

Future works could be devoted to further improving the computational speed up as to ensure that the efficiency in computing the stress field can directly transfer to the displacement $u$ and the rotation $r$ as well.

\appendix

\section{Principal Orthogonal Decomposition (POD)} \label{appendix: POD}
Let $\{\xi_i\}_{i=1}^{N}\subset H$ be a point cloud in $\mathbb{R}^{N_h}.$ Proper Orthogonal Decomposition is an algorithm that, given a reduced dimension $n\ll \min\{N,N_h\}$, seeks for the matrix $\mathsf{V}_n\in\mathbb{R}^{N_h\times n}$ that minimizes the mean squared projection error, i.e.
\begin{equation}
\label{eq: proj error}
\mathsf{V}_n:=\argmin_{\mathsf{W}\in\mathbb{R}^{N_h\times n}}\;\frac{1}{N}\sum_{i=1}^{N}\|\xi_i-\mathsf{W}\mathsf{W}^\top \xi_i\|^2.\end{equation}
The solution to the latter minimization problem is known in closed form and it can be computed via truncated Singular Value Decomposition (SVD). More precisely, let 
$$\Xi:=[\xi_1,\dots,\xi_N]\in\mathbb{R}^{N_h\times N},$$
be the so-called snapshots matrix. Computing an SVD of the above yields
\begin{equation*}
\Xi=\mathsf{V}\mathbf{\Sigma}\mathsf{U}^\top,
\end{equation*}
for suitable orthonormal matrices $\mathsf{U},\mathsf{V}\in\mathbb{R}^{N\times r}$, and a diagonal matrix $\mathbf{\Sigma}\in\mathbb{R}^{r\times r}$. Without loss of generality, we assume the entries $\mathbf{\Sigma}=\text{diag}(s_1,\dots,s_r)$ to be sorted such that $s_1\ge\dots \ge s_r\ge0$. Here, $r:=\text{rank}(\Xi)$.

Then, the POD algorithm finds $\mathsf{V}_n$ by extracting the submatrix corresponding to the first $n$ columns of $\mathsf{V}$. It can be proven that this procedure actually results in a solution to \eqref{eq: proj error}: see, e.g., \cite[Proposition 6.1]{quarteroni2016reduced}. We mention that, given any positive definite matrix $\mathsf{G}\in\mathbb{R}^{N_h\times N_h}$, one can replace the $\ell^2$-norm in \eqref{eq: proj error} with the weighted norm $\|\cdot\|_{\mathsf{G}}$ induced by the inner-product,
$$\langle \xi, \eta\rangle_{\mathsf{G}}:=\xi^\top\mathsf{G}\eta.$$
In practice, this boils down to computing an SVD of $\mathsf{G}^{1/2}\Xi$ and re-adapting the previous ideas: cf. \cite[Proposition 6.2]{quarteroni2016reduced}.

In both cases, the purpose of the POD is to construct a matrix spanning a subspace of dimension $n$ that approximates the point cloud as accurately as possible. In fact, notice that for every $\xi_i$ in the point cloud, due optimality of orthogonal projections,
$$\|\xi_i-\mathsf{V}_n\mathsf{V}_n^\top\xi_i\|^2=\min_{\eta\in\text{span}(\mathsf{V}_n)}\;\|\xi_i-\eta\|^2.$$

\section{Neural network architectures} \label{appendix: architectures}

In this appendix, we provide additional details on the neural network architectures implemented for the case studies in \Cref{sec:experiments}. In what follows, given two positive integers $p,k\in\mathbb{N}_{+}$, we denote by $F_{p,k}$ the map from $\mathbb{R}^{p}\to\mathbb{R}^{2pk+p}$ acting as
\begin{equation*}
    F_{p,k}:
    \left[\begin{array}{c}
         x_1\\
         x_2\\
         \vdots\\
         x_p 
    \end{array}\right]\mapsto
    \left[\begin{array}{c}
         x_1\\
         \cos(x_1)\\
         \sin(x_1)\\
         \vdots\\
         \cos(kx_1)\\
         \sin(kx_1)
         \end{array}\right]
    \oplus\left[\begin{array}{c}
         x_2\\
         \cos(x_2)\\
         \sin(x_2)\\
         \vdots\\
         \cos(kx_2)\\
         \sin(kx_2)
         \end{array}\right]
    \oplus\dots\oplus
    \left[\begin{array}{c}
         x_p\\
         \cos(x_p)\\
         \sin(x_p)\\
         \vdots\\
         \cos(kx_p)\\
         \sin(kx_p)
         \end{array}\right].
\end{equation*}
where $\oplus$ is the concatenation operator, acting on pair of vectors $\mathbf{y}=[y_1,\dots,y_a]^\top$, $\mathbf{z}=[y_1,\dots,y_b]^\top$ as $\mathbf{y}\oplus\mathbf{z}:=[y_1,\dots,y_a,z_1,\dots,z_b]^\top$.
Essentially, $F_{p,k}$ acts as a feature augmentation map, adding sinusoidal transformations of each $x_i$ up to frequency $k$. These maps are commonly used in the Deep Learning literature as a form of pre-processing and are sometimes referred to as \emph{Fourier layers} (not to be confused with Fourier layers in Neural Operators \cite{kovachki2023neural}, nor with classical dense layers: the ones discussed here, in fact, are \emph{nontrainable}, as their definition does not include learnable parameters).

Conversely, following \Cref{def:layer}, we denote by $D_{a,b}^{\rho}$ a generic dense feed forward layer from $\mathbb{R}^{a}\to\mathbb{R}^{b}$ with 0.1-leakyReLU activation, that is, relying upon
\begin{equation*}
    \rho(x):=\begin{cases}
        x & x\ge0,\\
        0.1x & x<0.
    \end{cases}
\end{equation*}
We also use the notation $D_{a,b}$, without superscript, to intend layers without activation.

With these convention, the architectures employed for the three case studies can be summarized with following scheme, shared by all models:
\begin{equation}\label{label:architecture}
D_{N_h,n}\circ D_{n,30}^{\rho}\circ D_{30,30}^{\rho}\circ D_{30,2pk+p}\circ F_{k,p},    
\end{equation}
where the only problem-dependent quantities are: $p$ the number of PDE parameters, $N_h$ the total number of dof in the stress field discretization, $n$ and $k$. The values of these quantities are summarized in \Cref{tab:numbers}.

More precisely, all models in \Cref{sec:experiments} referred to as "Black-box" implement the architecture in Eq. \eqref{label:architecture}. All the other approaches (POD-NN, Split, Corrected), instead, implement \eqref{label:architecture} using the POD-NN strategy. In particular, the terminal layer $D_{N_{h},n}$ is replaced by the POD matrix $\mathbf{V}\in\mathbb{R}^{N_h\times n}$, acting linearly, and the activation of the second-last layer is removed, formally replacing $D_{n,30}^{\rho}$ with $D_{n,30}.$ Consequently, for these architectures, $n$ can be interpreted as the POD dimension (or, equivalently, the number of POD modes).

\begin{table}[ht]
    \caption{Number of parameters, $p$, number of dof, $N_h=\dim(\Sigma_h)$, POD dimension, $n$, and Fourier frequency, $k$, across the three case studies in \Cref{sec:experiments}.\\}
    \label{tab:numbers}
    \centering
    \begin{tabular}{lllll}
    \hline\hline
     \textbf{Case study}  & $p$ & $N_h$ & $n$ & $k$\\
     \hline
     \S\ref{subsec:ex1} Footing problem     & 4 & 5824 & 10 & 3\\
     \S\ref{subsec:ex2} Cantilever  & 3 & 9018 & 7 & 3\\
     \S\ref{subsec:ex3} Hencky-von Mises & 4 & 5824 & 15 & 2\\\hline\hline
    \end{tabular}
\end{table}

\section*{Acknowledgements}

AF has been partially funded by the PRIN project ``FREYA - Fault REactivation: a hYbrid numerical Approach'' - SLG2RIST01. NF has been supported by project Cal.Hub.Ria (Piano Operativo Salute, traiettoria 4), funded by MSAL. AF and NF are both members of "Gruppo Nazionale per il Calcolo Scientifico" (GNCS).
The present research is part of the activities of ``Dipartimento di Eccellenza 2023-2027'', Italian Minister of University and Research (MUR), grant Dipartimento di Eccellenza 2023-2027.

\bibliographystyle{siam}
\bibliography{references}
\end{document}